\documentclass[12pt,a4paper]{article}
\usepackage{amsmath,amstext,amssymb,amscd, color}
\usepackage{graphicx}

\newcommand{\Xcomment}[1]{}

\newtheorem{theorem}{Theorem}[section]

\newtheorem{prop}[theorem]{Proposition}

\newcommand{\SEC}[1]{\ref{sec:#1}}  
\newcommand{\SSEC}[1]{\ref{ssec:#1}}  
\newcommand{\THEO}[1]{\ref{tm:#1}}  

\makeatletter \@addtoreset{equation}{section} \makeatother

\newenvironment{proof}{\noindent{\bf Proof}~}%
{\hfill$\qed$\medskip}

\def\qed{ \ \vrule width.1cm height.3cm depth0cm}

\newenvironment{numitem}{\refstepcounter{equation}\begin{enumerate}%
\item[(\thesection.\arabic{equation})]$\quad$}{\end{enumerate}}
\newenvironment{numitem1}{\refstepcounter{equation}\begin{enumerate}%
\item[(\thesection.\arabic{equation})]}{\end{enumerate}}

\newcommand{\refeq}[1]{(\ref{eq:#1})}  

 \makeatletter
\renewcommand{\section}{\@startsection{section}{1}{0pt}%
{-3.5ex plus -1ex minus -.2ex}{2.3ex plus .2ex}%
{\normalfont\Large}}
 \makeatother

 \makeatletter
\renewcommand{\subsection}{\@startsection{subsection}{2}{0pt}%
{-3.0ex plus -1ex minus -.2ex}{1.5ex plus .2ex}%
{\normalfont\normalsize\bf}}
 \makeatother

 \makeatletter
\renewcommand{\subsubsection}{\@startsection{subsubsection}{2}{0pt}%
{-2.0ex plus -1ex minus -.2ex}{-2.0ex plus .2ex}%
{\normalfont\normalsize\underline}}
 \makeatother

\def\Rset{{\mathbb R}}
\def\Zset{{\mathbb Z}}

\def\Bscr{{\cal B}}
\def\Cscr{{\cal C}}
\def\Dscr{{\cal D}}
\def\Escr{{\cal E}}
\def\Fscr{{\cal F}}

\def\Hscr{{\cal H}}
\def\Iscr{{\cal I}}

\def\Nscr{{\cal N}}
\def\Pscr{{\cal P}}

\def\Sscr{{\cal S}}
\def\Tscr{{\cal T}}

\def\Xscr{{\cal X}}

\def\Inv{{\rm Inv}}
\def\ident{{\rm id}}

\def\tilde{\widetilde}
\def\hat{\widehat}

\def\eps{\epsilon}

\def\Rin{R_\Sscr^{{\rm in}}}
\def\Rout{R_\Sscr^{{\rm out}}}
\def\Din{\Dscr_\Sscr^{{\rm in}}}
\def\Dout{\Dscr_\Sscr^{{\rm out}}}
\def\Kin{K^{{\rm in}}}
\def\Kout{K^{{\rm out}}}
\def\Kpin{K'^{\,{\rm in}}}
\def\Kpout{K'^{\,{\rm out}}}

\def\Olow{O^{{\rm low}}}
\def\Oup{O^{{\rm up}}}

\def\Sigmalow{\Sigma^{{\rm low}}}
\def\Sigmaup{\Sigma^{{\rm up}}}

\def\Pleft{P^{{\rm left}}}
\def\Pright{P^{{\rm right}}}
\def\Zleft{Z^{{\rm left}}}
\def\Zright{Z^{{\rm right}}}
\def\Gleft{G^{{\rm left}}}
\def\Gright{G^{{\rm right}}}
\def\Kleft{K^{{\rm left}}}
\def\Kright{K^{{\rm right}}}

\def\strong{\hbox{\unitlength=1mm\begin{picture}(8,4)%
\put(3.5,1){\oval(6,3)}\put(1.2,0){{\small str}}\end{picture}}}

\def\weak{\hbox{\unitlength=1mm\begin{picture}(10.5,4)%
\put(5,1){\oval(9,3)}\put(1.2,0){{\small weak}}\end{picture}}}

\begin{document}

 \title{Combined tilings and separated set-systems}

 \author{V.I.~Danilov\thanks{Central Institute of Economics and
Mathematics of the RAS, 47, Nakhimovskii Prospect, Moscow, Russia 117418;
email: danilov@cemi.rssi.ru.}
 \and
A.V.~Karzanov\thanks{Institute for System Analysis at FRC Computer Science and
Control of the RAS, 9, Prospect 60 Let Oktyabrya, Moscow, Russia 117312; email:
sasha@cs.isa.ru. Corresponding author.}
  \and
G.A.~Koshevoy\thanks{Central Institute of Economics and Mathematics of the RAS,
47, Nakhimovskii Prospect, Moscow, Russia 117418; email:
koshevoy@cemi.rssi.ru.}
 }

 \date{}

 \maketitle

 \begin{quote}
 {\bf Abstract.} \small
In 1998, Leclerc and Zelevinsky introduced the notion of \emph{weakly
separated} collections of subsets of the ordered $n$-element set $[n]$ (using
this notion to give a combinatorial characterization for quasi-commuting minors
of a quantum matrix). They conjectured the purity of certain natural domains
$\Dscr\subseteq 2^{[n]}$ (in particular, of the hypercube $2^{[n]}$ itself, and
the hyper-simplex $\{X\subseteq[n]\colon |X|=m\}$ for $m$ fixed), where $\Dscr$
is called \emph{pure} if all maximal weakly separated collections in $\Dscr$
have the same cardinality. These conjectures have been answered affirmatively.

In this paper, generalizing those earlier results, we reveal wider classes of
pure domains in $2^{[n]}$. This is obtained as a consequence of our study of a
novel geometric--combinatorial model for weakly separated set-systems,
so-called \emph{combined (polygonal) tilings} on a zonogon, which yields a new
insight in the area.


 \end{quote}

\parskip=3pt

\section{Introduction}  \label{sec:intr}

For a positive integer $n$, the set $\{1,2,\ldots,n\}$ with the usual order is
denoted by $[n]$. For a set $X\subseteq[n]$ of elements $x_1<x_2<\cdots<x_k$,
we use notation $(x_1,\ldots,x_k)$ for $X$, $\min(X)$ for $x_1$, and $\max(X)$
for $x_k$, where $\min(X)=\max(X):=0$ if $X=\emptyset$.

We will deal with several binary relations on the set $2^{[n]}$ of all subsets
of $[n]$. Namely, for distinct $A,B\subseteq[n]$, we write:
  \begin{numitem1}
  \begin{itemize}
\item[(i)] $A\prec B$ if $A=(a_1,\ldots,a_k)$,
$B=(b_1,\ldots,b_m)$, $k\le m$, and $a_i\le b_i$ for $i=1,\ldots,k$
(\emph{termwise dominating});
\item[(ii)]  $A<B$ if $\max(A)<\min(B)$ (\emph{global dominating});
\item[(iii)]
$A\lessdot B$ if $(A-B)<(B-A)$, where $A'-B'$ denotes $\{i'\colon A'\ni
i'\not\in B'\}$ (\emph{global dominating after cancelations});
\item[(iv)]  $A\rhd B$ if $A-B\ne\emptyset$, and $B-A$
can be expressed as a union of non-empty subsets $B',B''$ so that
$B'<(A-B)<B''$ (\emph{splitting}).
  \end{itemize}
  \label{eq:2relat}
   \end{numitem1}
(Note that relations (i),(ii) are transitive, whereas (iii),(iv) are not in
general.) Relations (iii) and (iv) give rise to two important notions
introduced by Leclerc and Zelevinsky~\cite{LZ} (where these notions appear in
characterizations of quasi-commuting flag minors of a generic $q$-matrix).
\medskip

\noindent \textbf{Definitions.} ~Sets $A,B\subseteq[n]$ are called
\emph{strongly separated} (from each other) if $A\lessdot B$ or $B\lessdot A$
or $A=B$. ~Sets $A,B\subseteq[n]$ are called \emph{weakly separated} if either
they are strongly separated, or $A\rhd B$ and $|A|\ge |B|$, or $B\rhd A$ and
$|B|\ge|A|$. Accordingly, a collection $\Fscr\subseteq 2^{[n]}$ is called
strongly (resp. weakly) separated if any two of its members are such. For
brevity, we refer to strongly and weakly separated collections as
\emph{s-collections} and \emph{w-collections}, respectively. \medskip

In what follows, we will distinguish one or another set-system $\Dscr\subseteq
2^{[n]}$, referring to it as a ground collection, or a \emph{domain}. We are
interested in the situation when (strongly or weakly) separated collections in
$\Dscr$ possess the property of \emph{max-size purity}, which means the
following.
\medskip

\noindent \textbf{Definitions.} ~Let us say that a domain $\Dscr$ is
\emph{s-pure} if all (inclusion-wise) maximal s-collections in $\Dscr$ have the
same cardinality, which in this case is called the \emph{s-rank} of $\Dscr$ and
denoted by $r^s(\Dscr)$. Similarly, we say that $\Dscr$ is \emph{w-pure} if all
maximal w-collections in $\Dscr$ have the same cardinality, called the
\emph{w-rank} of $\Dscr$ and denoted by $r^w(\Dscr)$.
\medskip

(In topology, the term ``pure'' is often applied to complexes in which all
maximal cells have the same dimension. In our case we can interpret each
s-collection (resp. w-collection) as a cell, forming an abstract simplicial
complex with $\Dscr$ regarded as the set of zero-dimensional cells. This
justifies our terms ``s-pure'' and ``w-pure''.)

Leclerc and Zelevinsky~\cite{LZ} proved that the maximal domain
(\emph{hypercube}) $\Dscr=2^{[n]}$ is s-pure and conjectured that $2^{[n]}$ is
w-pure as well (in which case there would be
$r^w(2^{[n]})=r^s(2^{[n]})=\frac{n(n+1)}{2}+1$). A sharper version of this
conjecture deals with $\omega$-\emph{chamber sets} $X\subseteq[n]$ for a
permutation $\omega$ on $[n]$, where $X$ obeys the condition:
  \begin{numitem}  \label{eq:omega_chamber}
  if $i<j$, ~$\omega(i)<\omega(j)$, and $j\in X$, then $i\in X$.
  \end{numitem}
They conjectured that the domain $\Dscr(\omega)$ consisting of the
$\omega$-chamber sets is w-pure (in our terms), with the w-rank equal to
$|\Inv(\omega)|+n+1$. Here $\Inv(\omega)$ denotes the set of \emph{inversions}
of $\omega$ (i.e., pairs $(i,j)$ in $[n]$ such that $i<j$ and
$\omega(i)>\omega(j)$), and the number $|\Inv(\omega)|$ is called the
\emph{length} of $\omega$. For the \emph{longest} permutation $\omega_0$ (where
$\omega_0(i)=n-i+1$), we have $\Dscr({\omega_0})=2^{[n]}$.

This conjecture was proved in~\cite{DKK2}. The key part consisted in proving
the w-purity for $2^{[n]}$; using this, the w-purity was shown for an arbitrary
permutation $\omega$, and more.

  \begin{theorem}[\cite{DKK2}] \label{tm:LZconj}
The hypercube $2^{[n]}$ is w-pure. As a consequence, the following domains
$\Dscr$ are w-pure as well:
  \begin{itemize}
\item[{\rm(i)}] $\Dscr=\Dscr(\omega)$ for any permutation
$\omega$ on $[n]$;
\item[{\rm(ii)}] $\Dscr=\Dscr({\omega',\omega})$, where
$\omega',\omega$ are two permutations on $[n]$ with $\Inv(\omega')\subset
\Inv(\omega)$, and $\Dscr({\omega',\omega})$ is formed by the $\omega$-chamber
sets $X\subseteq [n]$ satisfying the additional condition: if $i<j$,
~$\omega'(i)>\omega'(j)$, and $i\in X$, then $j\in X$; furthermore,
$r^w(\Dscr({\omega,\omega'}))=|Inv(\omega)|-|Inv(\omega')|+n+1$;
\item[{\rm(iii)}] $\Dscr=\Delta_n^{m',m}:=
\{X\subseteq[n]\colon m'\le |X|\le m\}$ for any $m'\le m$; furthermore,
$r^w(\Delta_n^{m',m})=\binom{n+1}{2}-\binom{n-m+1}{2}-\binom{m'+1}{2}+1$ (this
turns into $m(n-m)+1$ when $m'=m$).
  \end{itemize}
  \end{theorem}

Note that (ii) generalizes (i) since $\Dscr(\omega)=\Dscr({\ident,\omega})$,
where $\ident$ is the identical permutation ($\ident(i)=i$). The domain
$\Delta_n^{m',m}$ in~(iii) generalizes the Boolean hyper-simplex, or
\emph{discrete Grassmannian}, $\Delta_n^m:=\Delta_n^{m,m}$. The domains in
cases (i) and (ii) turn out to be s-pure as well, and the w- and s-ranks are
equal; see~\cite{DKK2}. (Note that in general a domain $\Dscr$ may be w-pure
but not s-pure (e.g. for $\Dscr=\Delta_5^2$), and vice versa; also when both w-
and s-ranks exist, they may be different.) Using simple observations
from~\cite{LZ}, one can reduce case (iii) to $2^{[n]}$ as well. In its turn,
the proof of w-purity for $2^{[n]}$ given in~\cite{DKK2} is direct and
essentially relies on a mini-theory of \emph{generalized tilings} developed
in~\cite{DKK1}.

Oh et al.~\cite{OPS} gave another proof for $\Delta_n^m$ using \emph{plabic
tilings} (and relying on a machinery of  \emph{plabic graphs} elaborated
in~\cite{Post}). Moreover, \cite{OPS} established the w-purity for certain
domains $\Dscr\subseteq\Delta_n^m$ related to so-called \emph{Grassmann
necklaces}. They also explained that the w-purity of such necklace domains
implies the w-purity for the $\omega$-chamber domain $\Dscr(\omega)$ with any
permutation $\omega$ (however, it is not clear whether the w-purity of
$\Dscr(\omega',\omega)$ can be obtained directly from results in~\cite{OPS}).

The purpose of this paper is twofold: to extend the above mentioned w-purity
results by demonstrating wider classes of domains whose w-purity follows from
the w-purity of $2^{[n]}$, and to describe a novel geometric--combinatorial
representation of maximal w-collections in $2^{[n]}$ (which is essentially used
to carry out the first task, but also is interesting by its own right and can
find other important applications).

More precisely, in the half-plane $\Rset\times\Rset_{\ge 0}$ we fix $n$ generic
vectors $\xi_1,\ldots,\xi_n$ having equal norms and ordered as indicated around
the origin. Under the linear projection $\Rset^{[n]}\to \Rset^2$ mapping $i$th
orth to $\xi_i$, the image of the solid cube $conv(2^{[n]})$ is a
\emph{zonogon} $Z$ (a central symmetric $2n$-gone), and we naturally identify
each subset $X\subseteq [n]$ with the point $\sum(\xi_i\colon i\in X)$ of $Z$.
A well-known fact is that the maximal s-collections in $2^{[n]}$ one-to-one
correspond to the vertex sets of \emph{rhombus tilings} on $Z$, i.e., planar
subdivisions of $Z$ into rhombi with the same side lengths (cf.~\cite{LZ} where
the correspondence is described in dual terms of pseudo-line arrangements). In
case of w-collections, the construction given in~\cite{DKK1} is more
sophisticated: each maximal w-collection one-to-one corresponds to a
\emph{generalized tiling} (briefly, \emph{g-tiling}), which is viewed as a
certain covering of $Z$ by rhombi which may overlap. So a g-tiling need not be
fully planar; also the maximal w-collection corresponding to the tiling is
represented by a certain subset of its vertices, but not necessarily the whole
vertex set.

The simpler model for w-collections elaborated in this paper deals with a sort
of polygonal complexes on $Z$; they are fully planar and their vertex sets are
exactly the maximal w-collections in $2^{[n]}$. We call them \emph{combined
tilings}, or \emph{c-tilings} for short (as they consist of tiles of two types:
``semi-rhombi'' and ``lenses'', justifying the term ``combined''). A nice
planar structure and appealing properties of c-tilings will enable us to reveal
new classes of w-pure domains.

A particular class (extending that in~\cite{OPS} and a more general
construction in~\cite{DKK3}) is generated by a \emph{cyclic pattern}. This
means a sequence $\Sscr=(S_1,S_2,\ldots,S_r=S_0)$ of different weakly separated
subsets of $[n]$ such that, for each $i$, $|S_i-S_{i-1}|\le 1$ and
$|S_{i-1}-S_i|\le 1$. We call $\Sscr$ \emph{simple} if $|S_{i-1}|\ne|S_i|$
holds for each $i$, and \emph{generalized} otherwise. Considering the members
of $\Sscr$ as the corresponding points in the zonogon $Z$, we connect the pairs
of consecutive points by straight-line segments, obtaining the piecewise linear
curve $\zeta_\Sscr$. We show that $\zeta_\Sscr$ is non-self-intersecting when
$\Sscr$ is simple, and give necessary and sufficient conditions on $\Sscr$ to
provide that $\zeta_\Sscr$ is non-self-intersecting in the generalized case.
This allows us to represent the collection $\Dscr_\Sscr$ of all subsets of
$[n]$ weakly separated from $\Sscr$ as the union of two domains $\Din$ and
$\Dout$ whose members (points) are located inside and outside $\zeta_\Sscr$,
respectively. Domains of type $\Din$ generalize the ones in
Theorem~\THEO{LZconj} and the ones described in the necklace constructions
of~\cite{OPS}.

We show that both domains $\Din$ and $\Dout$ are w-pure. Moreover, these
domains form a \emph{complementary pair}, in the sense that any $X\in\Din$ and
any $Y\in\Dout$ are weakly separated from each other (proving a conjecture on
generalized necklaces in~\cite{DKK3}). The latter property follows from the
observation that for any two c-tilings $T$ and $T'$ on $Z$ whose vertex sets
contain the given cyclic pattern $\Sscr$, one can extract and exchange their
portions lying inside $\zeta_\Sscr$, obtaining two correct c-tilings again.

The most general case of w-pure domains that we describe is obtained when the
role of a pattern is played by a \emph{planar graph} $\Hscr$, embedded in a
zonogon, whose vertices form a w-collection $\Sscr$, and edges are given by
pairs similar to those in cyclic patterns. Then, instead of two domains
$\Din,\Dout$ as above, we deal with a set of domains $\Dscr^F_\Sscr$, each
being associated with a face $F$ of $\Hscr$. We prove that any two of them form
a complementary pair. It follows that each $\Dscr^F_\Sscr$ is w-pure, as well
as any union of them.

Note that for cyclic patterns consisting of strongly separated sets, we can
consider analogous pairs of domains w.r.t. the strong separation relation. One
shows that such domains are s-pure (which is relatively easy) and that the
corresponding s-ranks and w-ranks are equal.
\smallskip

This paper is organized as follows. Section~\SEC{cyclic} considers simple
cyclic patterns $\Sscr$ and related domains $\Din$ and $\Dout$ and state a
w-purity result for them (Theorem~\ref{tm:cyc_pattern}). Section~\SEC{tilings}
starts with a brief review on rhombus tilings and then introduces combined
tilings. The most attention in this section is drawn to transformations of
c-tilings, so-called \emph{raising} and \emph{lowering flips}, which correspond
to standard mutations for maximal w-collections. As a result, we establish a
bijection between the c-tilings on the zonogon and the maximal w-collections in
$2^{[n]}$. Section~\SEC{proofs} further develops our ``mini-theory'' of
c-tilings by describing operations of \emph{contraction} and \emph{expansion}
(analogous to those elaborated for g-tilings in~\cite[Sec.~8]{DKK1} but looking
more transparent for c-tilings). These operations transform c-tilings on the
$n$-zonogon to ones on the $(n-1)$-zonogon, and conversely. This enables us to
conduct needed proofs by induction on $n$, and using this and results from
Sect.~\SEC{tilings}, we finish the proof of Theorem~\ref{tm:cyc_pattern}. The
concluding Sect.~\SEC{gen} is devoted to special cases, illustrations and
generalizations. Here we extend the w-purity result to generalized cyclic
patterns (Theorem~\ref{tm:gen_cyc_pattern}), and finish with planar graph
patterns (Theorem~\ref{tm:gen-gen}).
\medskip

Note that the preliminary version~\cite{DKK4} of this paper gave one more
combinatorial construction of w-pure domains (based on lattice paths in certain
ladder diagrams), which is omitted here to make our description more compact.
 \medskip

 \noindent
\textbf{Additional notations.}
An \emph{interval} in $[n]$ is a set of the form $\{p,p+1,\ldots,q\}$, and a
\emph{co-interval} is the complement of an interval to $[n]$. For $p\le q$, we
denote by $[p..q]$ the interval $\{p,p+1,\ldots,q\}$.

When sets $A,B\subseteq[n]$ are weakly (strongly) separated, we write $A\weak
B$ (resp. $A\strong B$). We use similar notation $A\weak \Bscr$ (resp.
$A\strong \Bscr$) when a set $A\subseteq[n]$ is weakly (resp. strongly)
separated from all members of a collection $\Bscr\subseteq 2^{[n]}$.

For a set $X\subset[n]$, distinct elements $i,\ldots,j\in[n]-X$, and an element
$k\in X$, we abbreviate $X\cup\{i\}\cup\ldots\cup\{j\}$ as $Xi\ldots j$, and
$X-\{k\}$ as $X-k$.

By a \emph{path} in a directed graph, we mean a sequence
$P=(v_0,e_1,v_1,\ldots,e_k,v_k)$, where each $e_i$ is an edge connecting
vertices $v_{i-1}$ and $v_i$. An edge $e_i$ is called \emph{forward}
(\emph{backward}) if it goes from $v_{i-1}$ to $v_i$ (resp.  from $v_i$ to
$v_{i-1}$), and we write $e_i=(v_{i-1},v_i)$ (resp. $e_i=(v_i,v_{i-1})$). The
path is called \emph{directed} if all its edges are forward. Sometimes we will
use for $P$ an abbreviated notation via vertices, writing $P=v_0v_1\ldots v_k$.
\smallskip

In what follows, we will use the following simple fact for $A,B\subset[n]$:
 \begin{numitem1}
~if $A\weak B$ and $|A|\le|B|$, then relations $A\prec B$ and $A\lessdot B$ are
equivalent.
 \label{eq:prec-lessdot}
  \end{numitem1}

Another useful property is as follows (a similar property is valid for the
strong separation as well).

  \begin{prop} \label{pr:complement}
Let domains $\Dscr,\Dscr'\subset 2^{[n]}$ be such that $X\weak Y$ for any $X\in
\Dscr$ and $Y\in\Dscr'$. Let $\Dscr\cup\Dscr'$ be w-pure. Then each of
$\Dscr,\Dscr'$ is w-pure.
  \end{prop}

Indeed, let $A:=\Dscr\cap\Dscr'$, fix a maximal w-collection $X$ in $\Dscr$,
and take an arbitrary maximal w-collection $Y$ in $\Dscr'$. It is easy to see
that $A=X\cap Y$ and that $X\cup Y$ is a maximal w-collection in
$\Dscr\cup\Dscr'$. We have $|Y|=|X\cup Y|+|A|-|X|$, and now the w-purity of
$\Dscr'$ follows from that of $\Dscr\cup\Dscr'$. The argument for $\Dscr$ is
similar.
\smallskip

We say that domains $\Dscr,\Dscr'$ as in this proposition form a
\emph{complementary pair}.


\section{Simple cyclic patterns and their geometric interpretation}  \label{sec:cyclic}

Recall that by a \emph{simple cyclic pattern} we mean a sequence $\Sscr$ of
subsets $S_1,S_2,\ldots,S_r=S_0$ of $[n]$ such that $|S_{i-1}\triangle S_i|=1$
for $i=1,\ldots,r$, where $A\triangle B$ denotes the symmetric difference
$(A-B)\cup(B-A)$. We assume that $\Sscr$ satisfies the following conditions:
\medskip

\noindent \textbf{(C1)} ~All sets in $\Sscr$ are different; \smallskip

\noindent \textbf{(C2)} ~$\Sscr$ is weakly separated. \smallskip

\noindent (Condition (C1) can be slightly weakened, as we explain in
Section~\SSEC{semi-simple}.) We associate with $\Sscr$ the collection
  $$
  \Dscr_\Sscr:=\{X\subseteq [n] : X \weak \Sscr\}.
  $$
We extract from $\Dscr_\Sscr$ two domains $\Din$ and $\Dout$, using a geometric
construction mentioned in Sect.~1.

More precisely, in the upper half-plane $\Rset\times \Rset_{\ge 0}$, we fix $n$
vectors $\xi_1,\ldots,\xi_n$ which:
  \begin{numitem1} \label{eq:xi_i}
 (a) follow in this order clockwise around the
origin $(0,0)$;

(b) have the \emph{same euclidean norm}, say, $\Vert\xi_i\Vert=1$; and

(c) are $\Zset$-independent (i.e., all integer combinations of these vectors
are different).
  \end{numitem1}
We call these vectors \emph{generators}. Then the set
  $$
Z=Z_n:=\{\lambda_1\xi_1+\cdots+ \lambda_n\xi_n\colon \lambda_i\in\Rset,\;
0\le\lambda_i\le 1,\; i=1,\ldots,n\}
  $$
is a $2n$-gon. Moreover, $Z$ is a {\em zonogon}, as it is the sum of line
segments $\{\lambda\xi_i\colon 0\le \lambda\le 1\}$, $i=1,\ldots,n$. Also it is
the projection of the solid $n$-cube $conv(2^{[n]})$ into the plane, given by
$\pi(x):=x_1\xi_1+\cdots +x_n\xi_n$, $x\in\Rset^{[n]}$. We identify a subset
$X\subseteq [n]$ with the corresponding point $\sum(\xi_i\colon i\in X)$ in
$Z$. Due to the $\Zset$-independence, all such points in $Z$ are different.

The \emph{boundary} $bd(Z)$ of $Z$ consists of two parts: The {\em left
boundary} $lbd(Z)$ formed by the sequence of points (vertices)
$z^\ell_i:=\xi_1+\cdots+\xi_i$ ($i=0,\ldots,n$) connected by the line segments
$z^\ell_{i-1}z^\ell_i$ congruent to $\xi_i$, and the {\em right boundary}
$rbd(Z)$ formed by the sequence of points
$z^r_i:=\xi_n+\xi_{n-1}+\cdots+\xi_{n-i+1}$ ($i=0,\ldots,n$) connected by the
line segments $z^r_{i-1}z^r_{i}$ congruent to $\xi_{n-i+1}$. Then
$z^\ell_0=z^r_0=(0,0)$ is the bottom(most) vertex, and
$z^\ell_n=z^r_n=\xi_1+\cdots+\xi_n$ is the top(most) vertex of $Z$.

Returning to $\Sscr$ as above, we draw the closed piecewise linear curve
$\zeta_\Sscr$ by concatenating the line-segments connecting consecutive points
$S_{i-1}$ and $S_i$ for $i=1,\ldots,r$. The following property is important.
  \begin{prop} \label{pr:Rin-Rout}
For a simple cyclic pattern $\Sscr$, the curve $\zeta_\Sscr$ is
non-self-intersecting, and therefore it divides the zonogon $Z$ into two closed
regions $\Rin$ and $\Rout$, where $\Rin\cap \Rout=\zeta_\Sscr$, ~$\Rin\cup
\Rout=Z$, and $bd(Z)\subseteq \Rout$.
   \end{prop}

\noindent (So $\Rin$ is a disk and $\Rout$ is viewed as a ``ring''.) This
enables us to define the desired domains:
  \begin{equation} \label{eq:Din-Dout}
\Din:=\Dscr_\Sscr\cap\Rin \qquad \mbox{and} \qquad \Dout:=\Dscr_\Sscr\cap
\Rout.
  \end{equation}

Proposition~\ref{pr:Rin-Rout} is proved in Sect.~\SEC{proofs} relying on
properties of combined tilings introduced in the next section. Moreover, using
such tilings, we will prove the following w-purity result (where a
complementary pair is defined in Sect.~1).
  \begin{theorem} \label{tm:cyc_pattern}
For a simple cyclic pattern $\Sscr$, the domains $\Din$ and $\Dout$ form a
complementary pair. As a consequence, both $\Din$ and $\Dout$ are w-pure.
 \end{theorem}

This theorem has a natural strong separation counterpart.
  \begin{theorem} \label{tm:strong_cycl}
Let $\Sscr$ be a cyclic pattern consisting of different pairwise strongly
separated sets in $[n]$. Define $\hat \Dscr_{\Sscr}^{\rm in}$, $\hat
\Dscr_{\Sscr}^{\rm out}$ as in~\refeq{Din-Dout} with $\Dscr_\Sscr$ replaced by
$\hat\Dscr_\Sscr:=\{X\subseteq [n] : X \strong \Sscr\}$. Then $X\strong Y$ for
any $X\in\hat \Dscr_{\Sscr}^{\rm in}$ and $Y\in \hat \Dscr_{\Sscr}^{\rm out}$.
Hence both $\hat \Dscr_{\Sscr}^{\rm in}$ and $\hat \Dscr_{\Sscr}^{\rm out}$ are
s-pure. Also the s- and w-ranks of $\hat\Dscr_{\Sscr}^{\rm in}$ are equal, and
similarly for $\hat\Dscr_{\Sscr}^{\rm out}$.
  \end{theorem}


\section{Tilings}  \label{sec:tilings}

We know that maximal s-collections (w-collections) are representable via
rhombus (resp. generalized) tilings on the zonogon $Z=Z_n$. In this section, we
start with a short review on pure (rhombus) tilings and then introduce the
notion of \emph{combined tilings} on $Z$ and show that the latter objects
behave similarly to g-tilings: They represent maximal w-collections. At the
same time, this new combinatorial model is viewed significantly simpler than
the one of g-tilings, and we will take advantage from this fact to obtain a
rather transparent proof of Theorem~\ref{tm:cyc_pattern}.


\subsection{Rhombus tilings}  \label{ssec:rhomb_til}

By a \emph{planar tiling} on the zonogon $Z$, we mean a collection $T$ of
convex polygons, called \emph{tiles}, such that: (i) the union of tiles is $Z$,
(ii) any two tiles either are disjoint or intersect by a common vertex or by a
common side, and (iii) each boundary edge of $Z$ belongs to exactly one tile.
Then the set of vertices and the set of edges (sides) of tiles in $T$, ignoring
multiplicities, form a planar graph, denoted by $G_T=(V_T,E_T)$. Usually the
edges of $G_T$ are equipped with directions. Speaking of vertices or edges of
$T$, we mean those in $G_T$.

A \emph{pure (rhombus) tiling} is a planar tiling $T$ in which all tiles are
rhombi; then each edge of $T$ is congruent to some generator $\xi_i$, and we
direct it accordingly (upward).  More precisely, each tile $\tau$ is of the
form $X+\{\lambda\xi_i+\lambda'\xi_j\colon 0\le \lambda,\lambda'\le 1\}$ for
some $i<j$ and some subset $X\subseteq[n]-\{i,j\}$ (regarded as a point in
$Z$). We call $\tau$ a tile of \emph{type} $ij$ and denote it by $\tau(X;i,j)$.
By a natural visualization of $\tau$, its vertices $X,Xi,Xj,Xij$ are called the
{\em bottom, left, right, top} vertices of $\tau$ and denoted by $b(\tau)$,
$\ell(\tau)$, $r(\tau)$, $t(\tau)$, respectively.

The vertex set $V_T$ of $G_T$ is also called the \emph{spectrum} of $T$.

Pure tilings admit a sort of \emph{mutations}, called strong raising and
lowering flips, which transform one tiling $T$ into another, and back. More
precisely, suppose that $T$ contains a hexagon $H$ formed by three tiles
$\alpha:=\tau(X;i,j)$, $\beta:=\tau(X;j,k)$, and $\gamma:=\tau(Xj;i,k)$. The
\emph{strong raising flip} replaces $\alpha,\beta,\gamma$ by the other possible
combination of three tiles in $H$, namely by $\alpha':=\tau(Xk;i,j)$,
$\beta':=\tau(Xi;j,k)$, and $\gamma':=\tau(X;i,k)$. The \emph{strong lowering
flip} acts conversely: it replaces $\alpha',\beta',\gamma'$ by
$\alpha,\beta,\gamma$. See the picture.

 \begin{center}
  \unitlength=1mm
  \begin{picture}(130,33)
   \put(10,0){\begin{picture}(50,30)
  \put(10,0){\circle*{1}}
  \put(0,10){\circle*{1}}
  \put(30,10){\circle*{1}}
  \put(0,20){\circle*{1}}
  \put(10,10){\circle*{1}}
  \put(30,20){\circle*{1}}
  \put(20,30){\circle*{1}}
  \put(10,0){\vector(-1,1){9.7}}
  \put(10,10){\vector(-1,1){9.7}}
  \put(30,20){\vector(-1,1){9.7}}
  \put(10,0){\vector(2,1){19.7}}
  \put(10,10){\vector(2,1){19.7}}
  \put(0,20){\vector(2,1){19.7}}
  \put(0,10){\vector(0,1){9.5}}
  \put(10,0){\vector(0,1){9.5}}
  \put(30,10){\vector(0,1){9.5}}
  \put(11,7){$Xj$}
  \put(5,-2){$X$}
  \put(-6,7){$Xi$}
  \put(31,8){$Xk$}
  \put(-6,22){$Xij$}
  \put(31,21){$Xjk$}
  \put(22,29){$Xijk$}
  \put(4,9){$\alpha$}
  \put(20,9){$\beta$}
  \put(15,19){$\gamma$}
  \put(45,17){\vector(1,0){25}}
  \put(70,14){\vector(-1,0){25}}
  \put(47,19){\rm{raising flip}}
  \put(47,10){\rm{lowering flip}}
     \put(80,0){\begin{picture}(50,30)
  \put(10,0){\circle*{1}}
  \put(0,10){\circle*{1}}
  \put(30,10){\circle*{1}}
  \put(0,20){\circle*{1}}
  \put(20,20){\circle*{1}}
  \put(30,20){\circle*{1}}
  \put(20,30){\circle*{1}}
  \put(10,0){\vector(-1,1){9.7}}
  \put(30,10){\vector(-1,1){9.7}}
  \put(30,20){\vector(-1,1){9.7}}
  \put(10,0){\vector(2,1){19.7}}
  \put(0,10){\vector(2,1){19.7}}
  \put(0,20){\vector(2,1){19.7}}
  \put(0,10){\vector(0,1){9.5}}
  \put(20,20){\vector(0,1){9.5}}
  \put(30,10){\vector(0,1){9.5}}
  \put(13,7){$\gamma'$}
  \put(9,19){$\beta'$}
  \put(23,19){$\alpha'$}
  \put(15,14){$Xik$}
    \end{picture}}

    \end{picture}}
  \end{picture}
   \end{center}

\noindent In terms of sets, a strong flip is applicable to a set-system when
for some $X$ and $i<j<k$, it contains six sets $X,Xi,Xk,Xij,Xjk,Xijk$ and one
more set $Y\in\{Xj,Xik\}$. The flip (``in the presence of six witnesses'', in
terminology of~\cite{LZ}) replaces $Y$ by the other member of $\{Xj,Xik\}$; the
replacement $Xj\rightsquigarrow Xik$ gives a raising flip, and
$Xik\rightsquigarrow Xj$ a lowering flip. Leclerk and Zelevinsky described
in~\cite{LZ} important properties of strongly separated set-systems. Among
those we use the following.

(i) If $\Fscr\subseteq 2^{[n]}$ is a maximal s-collection and if $\Fscr'$ is
obtained by a strong flip from $\Fscr$, then $\Fscr'$ is a maximal s-collection
as well.

(ii) Let $\bf{S}_n$ be the set of maximal s-collections in $2^{[n]}$, and for
$\Fscr,\Fscr'\in\bf{S}_n$, let us write $\Fscr\prec_s \Fscr'$ if $\Fscr'$ is
obtained from $\Fscr$ by a series of strong raising flips. Then the poset
$(\bf{S}_n,\prec_s)$ has a unique minimal element and a unique maximal element,
which are the collection $\Iscr_n$ of all intervals and the collection
co-$\Iscr_n$ of all co-intervals in $[n]$, respectively.

Moreover, they established a correspondence between maximal s-collections and
so-called commutation classes of pseudo-line arrangements. In terms of pure
tilings (which are dual to pseudo-line arrangements, in a sense), this is
viewed as follows.
  \begin{theorem} {\rm \cite{LZ}} \label{tm:LZ_strong}
For each pure (rhombus) tiling $T$, its spectrum $V_T$ is a maximal
s-collection (regarding a vertex as a subset of $[n]$). This correspondence
gives a bijection between the set of pure (rhombus) tilings on $Z_n$ and
$\bf{S}_n$.
  \end{theorem}


\subsection{Combined tilings}  \label{ssec:combi}

In addition to the generators $\xi_i$ (satisfying~\refeq{xi_i}), we will deal
with vectors $\eps_{ij}:= \xi_j-\xi_i$, where $1\le i<j\le n$. A \emph{combined
tiling}, abbreviated as a \emph{c-tiling}, or simply a \emph{combi}, is a sort
of planar tilings (defined later) generalizing rhombus tilings in essence in
which each edge $e$ is congruent to either $\xi_i$ for some $i$ or $\eps_{ij}$
for some $i<j$; we say that $e$ has \emph{type} $i$ in the former case and
\emph{type} $ij$ in the latter case.
\medskip

\noindent\textbf{Remark 1.} In fact, we are free to choose arbitrary basic
vectors $\xi_i$ (subject to~\refeq{xi_i}) without affecting our model (in the
sense that corresponding structures and set-systems remain equivalent when the
generators vary). Sometimes, to simplify visualization, it is convenient to
think of vectors $\xi_i=(x_i,y_i)$  as ``almost vertical'' ones (i.e., to
assume that $|x_i|=o(|y_i|)$). Then the vectors $\eps_{ij}$ become ``almost
horizontal''. For this reason, we will liberally refer to edges congruent to
$\xi_\bullet$ as \emph{V-edges}, and those congruent to $\eps_\bullet$ as
\emph{H-edges}. Also we say that V-edges point \emph{upward}, and H-edges point
\emph{to the right}.
\medskip

The simplest case of combies arises from arbitrary rhombus tilings $T$ by
subdividing each rhombus $\tau$ of $T$ into two isosceles triangles $\Delta$
and $\nabla$, where the former (the ``upper'' triangle) uses the vertices
$\ell(\tau),t(\tau),r(\tau)$ and the latter (the ``lower'' triangle) uses the
vertices $\ell(\tau),b(\tau),r(\tau)$. Then the resulting combi has as V-edges
all edges of $T$ and has as H-edges the diagonals $(\ell(\tau),r(\tau))$ of
rhombi $\tau$ of $T$. We refer to such a combi as a \emph{semi-rhombus tiling}.

\vspace{0cm}
\begin{center}
\includegraphics{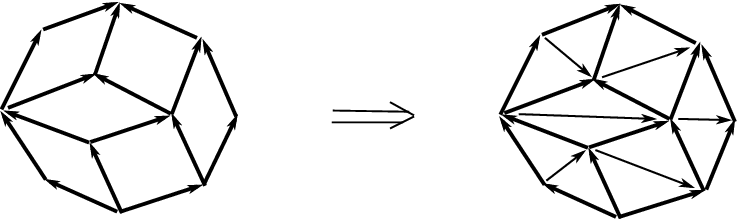}
\end{center}
\vspace{0cm}

The above picture illustrates the transformation of a rhombus tiling into a
combi; here V-edges (H-edges) are drawn by thick (resp. thin) lines.

In a general case, a combi is a planar tiling $K$ formed by tiles of three
sorts: $\Delta$-tiles, $\nabla$-tiles, and lenses. As before, the vertices of
$K$ (or $G_K$) represent subsets of $[n]$.

A $\Delta$-\emph{tile} is an isosceles triangle $\Delta$ with vertices $A,B,C$
and edges $(B,A),(C,A),(B,C)$, where $(B,C)$ is an H-edge, while $(B,A)$ and
$(C,A)$ are V-edges (so $(B,C)$ is the \emph{base} side and $A$ is the
\emph{top} vertex of $\Delta$). We denote $\Delta$ as $\Delta(A|BC)$.
\smallskip

A $\nabla$-\emph{tile} is symmetric. It is an isosceles triangle $\nabla$ with
vertices $A',B',C'$ and edges $(A',B'),(A',C'), (B',C')$, where $(B',C')$ is an
H-edge, while $(A',B')$ and $(A',C')$ are V-edges (so $(B',C')$ is the base
side and $A'$ is the \emph{bottom} vertex of $\nabla$). We denote $\nabla$ as
$\nabla(A'|B'C')$. The picture illustrates $\Delta$- and $\nabla$-tiles.

\vspace{-0.1cm}
\begin{center}
\includegraphics{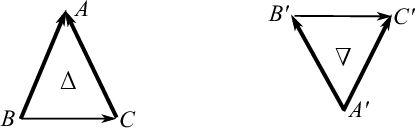}
\end{center}
\vspace{-0.1cm}

We say that a $\Delta$- or $\nabla$-tile $\tau$ has type $ij$ if its base edge
has this type (and therefore the V-edges of $\tau$ have types $i$ and $j$).
\smallskip

In a \emph{lens} $\lambda$, the boundary is formed by two directed paths
$U_\lambda$ and $L_\lambda$ with at least two edges in each. The paths
$U_\lambda$ and $L_\lambda$ have the same beginning vertex $\ell_\lambda$ and
the same end vertex $r_\lambda$ (called the \emph{left} and \emph{right}
vertices of $\lambda$, respectively), contain merely H-edges, and form the
\emph{upper} and \emph{lower} boundaries of $\lambda$, respectively. More
precisely,
  \begin{numitem1} \label{eq:U_lambda}
the path $U_\lambda=(\ell_\lambda=X_0,e_1,X_1,\ldots,e_q,X_q=r_\lambda)$ (where
$q\ge 2$) is associated with a sequence $i_0<i_1<\cdots<i_q$ in $[n]$ so that
for $p=1,\ldots,q$, the edge $e_p$ (going from the vertex $X_{p-1}$ to the
vertex $X_i$) is of type $i_{p-1}i_p$.
  \end{numitem1}
This implies that the vertices of $U_\lambda$ are expressed as $X_p=Xi_p$ for
one and the same set $X\subset[n]$ (equal to $X_p\cap X_{p'}$ for any $p\ne
p'$). In other words, the point $X_p$ in $Z$ is obtained from $X$ by adding the
vector $\xi_{i_p}$, and all vertices of $U_\lambda$ lie on the upper half of
the circumference of radius 1 centered at $X$. We call $X$ the \emph{center} of
$U_\lambda$ and denote by $\Oup_\lambda$. In its turn,
  \begin{numitem1} \label{eq:L_lambda}
the path
$L_\lambda=(\ell_\lambda=X'_0,e'_1,X'_1,\ldots,e'_{q'},X'_{q'}=r_\lambda)$
(with $q'\ge 2$) is associated with a sequence $j_0>j_1>\cdots>j_{q'}$ in $[n]$
so that for $p=1,\ldots,q'$, the edge $e'_p$ is of type $j_pj_{p-1}$.
  \end{numitem1}
Then the vertices of $L_\lambda$ are expressed as $X'_p=Y-j_p$ for the same set
$Y\subseteq[n]$ (equal to $X'_p\cup X'_{p'}$ for any $p\ne p'$). So the
vertices of $L_\lambda$ lie on the lower half of the circumference of radius 1
centered at $Y$, called the \emph{center} of $L_\lambda$ and denoted by
$\Olow_\lambda$. See the picture (where $q=2$ and $q'=3$):

\vspace{-0.3cm}
\begin{center}
\includegraphics{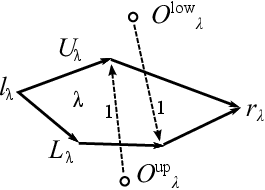}
\end{center}
\vspace{-0cm}

Note that $\ell_\lambda=Xi_0=Y-j_0$ and $r_\lambda=Xi_q=Y-j_{q'}$ imply
$i_0=j_{q'}$ and $i_q=j_0$. We say that the lens $\lambda$ \emph{has type}
$i_0i_q$. The intersection of the circles of radius 1 centered at
$\Oup_\lambda$ and $\Olow_\lambda$ is called the \emph{rounding} of $\lambda$
and denoted by $\Omega_\lambda$. Observe that all vertices of $\lambda$ lie on
the boundary of $\Omega_\lambda$. It is not difficult to realize that none of
the vertices of the combi $K$ lies in the interior (i.e., strictly inside) of
$\Omega_\lambda$.

Clearly all vertices of a lens $\lambda$ have the same size (regarding a vertex
as a subset of $[n]$); we call it the \emph{level} of $\lambda$. For an H-edge
$e$, consider the pair of tiles in $K$ containing $e$. Then either both of them
are lenses, or one is a lens and the other is a triangle, or both are triangles
($\Delta$- and $\nabla$-tiles); in the last case, we regard $e$ as a
\emph{degenerate lens}.

The union of lenses of level $h$ (including degenerate lenses) forms a closed
simply connected region meeting $lbd(Z)$ at the vertex $z^\ell_h=[h]$, and
$rbd(Z)$ at the vertex $z^r_h=[(n-h+1)..n]$; we call it $h$th \emph{girdle} and
denote by $\Lambda_h$. A vertex $v$ in $\Lambda_h$ having both entering and
leaving V-edges is called \emph{critical}; it splits the girdle into two (left
and right) closed sets, and the part of $\Lambda_h$ between two consecutive
critical vertices is either a single H-edge or a disk (being the union of some
non-degenerate lenses). The region between the upper boundary of $\Lambda_h$
and the lower boundary of $\Lambda_{h+1}$ is filled up by triangles; we say
that these triangles have level $h+\frac12$.

Figure~\ref{fig:combi} illustrates a combi for $n=4$ having one lens $\lambda$;
here the bold (thin) arrows indicate V-edges (resp.H-edges); note that the
girdle $\Lambda_2$ is the union of the lens $\lambda$ and the edge $(24,34)$.

\begin{figure}[htb]
\vspace{0cm}
\begin{center}
\includegraphics{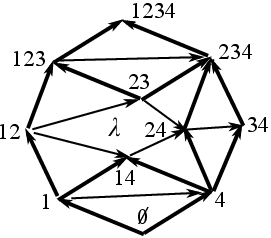}
\end{center}
\vspace{-0.5cm}
 \caption{The combi $K$ with $V_K=\{\emptyset,1,4,12,14,23,24,34,123,234,1234\}$}
 \label{fig:combi}
  \end{figure}

\noindent\textbf{Remark 2.} Fix a girdle $\Lambda_h$ of $K$ and transform each
lens $\lambda$ in it by drawing the line segment between $\ell_\lambda$ and
$r_\lambda$, thus subdividing $\lambda$ into two ``semi-lenses'', the upper and
lower ones. Let us color each of the obtained upper semi-lenses in black and
each lower semi-lens in white. Then the resulting planar tiling (within the
region of $\Lambda_h$), in which the tiles are colored black and white and any
two tiles sharing an edge have different colors, is viewed as a \emph{plabic
tiling}. Recall that plabic tilings are studied in~\cite{OPS} where it is
shown, in particular, that the set of vertices of such a tiling represents a
weakly separated collection in the discrete Grassmannian $\Delta_n^h$ for
corresponding $n,h$. A description of plabic tilings and their properties in
more details is beyond our paper.


\subsection{Flips in c-tilings}  \label{ssec:flips}

Now our aim is to show that the spectrum (viz. vertex set) $V_K$ of a combi $K$
is a maximal w-collection and that any maximal w-collection is obtained in this
way. To show this, we elaborate a technique of flips on combies (which, due to
the planarity of a combi, looks simpler than a technique of this sort for
g-tilings in~\cite{DKK1}).

We will rely on the following two statements.
  \begin{prop} \label{pr:X-Xi}
Suppose that a combi $K$ contains two vertices of the form $X$ and $Xi$ (where
$i\in[n]-X$). Then $K$ has edge from $X$ to $Xi$.
  \end{prop}
  \begin{prop} \label{pr:2edges}
Suppose that for some $i<j''\le j'<k$, a combi $K$ contains an H-edge $e=(A,B)$
of type $j'k$ and an H-edge $e'=(B,C)$ of type $ij''$. Then $j'=j''$ and the
edges $e,e'$ belong to either the lower boundary of one lens of $K$, or two
$\Delta$-tiles with the same top vertex, namely $\Delta(D|AB)$ and
$\Delta(D|BC)$, where $D=Ak=Bj'=Ci$. Symmetrically, if $e=(A,B)$ has type
$ij''$ and $e'=(B,C)$ has type $j'k$, then $j'=j''$ and $e,e'$ belong to either
the upper boundary of one lens of $K$, or two $\nabla$-tiles with the same
bottom vertex, namely $\nabla(D'|AB)$ and $\nabla(D'|BC)$, where
$D'=A-i=B-j'=C-k$.
  \end{prop}

These propositions are proved in Sect.~\SEC{proofs}. Assuming their validity,
we show the following.
  \begin{theorem} \label{tm:combi-max-ws}
For a combi $K$, the spectrum $V_K$ is a maximal w-collection.
  \end{theorem}
  \begin{proof}
First of all, we introduce W- and M-configurations in $K$ (their counterparts
for g-tilings are described in~\cite[Sec.~4]{DKK1}).

A \emph{W-configuration} is formed by two $\nabla$-tiles
$\nabla'=\nabla(Y'|X'X)$ and $\nabla''=\nabla(Y''|XX'')$ (resembling letter W),
as indicated in the left fragment of the picture.

\vspace{-0.3cm}
\begin{center}
\includegraphics{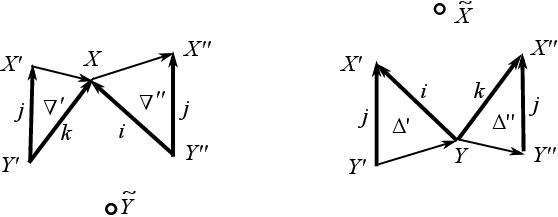}
\end{center}
\vspace{-0.3cm}

\noindent Here for some $i<j<k$, the left edge $(Y',X')$ of $\nabla'$ and the
right edge $(Y'',X'')$ of $\nabla''$ have type $j$, the right edge $(Y',X)$ of
$\nabla'$ has type $k$, and the left edge $(Y'',X)$ of $\nabla''$ has type $i$.
In other words, letting $\tilde Y:=Y'\cap Y''$ (which need not be a vertex of
$K$), we have
  \begin{equation} \label{eq:Wconf}
  Y'=\tilde Yi,\quad Y''=\tilde Yk,\quad X'=\tilde Yij,\quad X=\tilde
  Yik,\quad  X''=\tilde Yjk,
  \end{equation}
and denote such a W-configuration as $W(\tilde Y; i,j,k)$.

In its turn, an M-configuration is formed by two $\Delta$-tiles
$\Delta'=\Delta(X'|Y'Y)$ and $\Delta''=\Delta(X''|YY'')$ in $K$ (resembling
letter M), as indicated in the right fragment of the above picture. Here for
$i<j<k$, the V-edges $(Y',X'),(Y,X'),(Y,X''),(Y'',X'')$ have types $j,i,k,j$,
respectively. For $\tilde Y:=Y'\cap Y''$ (as before), the vertices
$Y',X',Y'',X''$ are expressed as in~\refeq{Wconf}, and
  \begin{equation} \label{eq:tildeYj}
  Y=\tilde Yj.
  \end{equation}
We denote such an M-configuration as $M(\tilde Y; i,j,k)$.

When $K$ has a W-configuration (M-configuration), we can make a \emph{weak
lowering} (resp. \emph{raising}) \emph{flip} to transform $K$ into another
combi $K'$. (This is somewhat similar to ``strong'' flips in pure tilings but
now the flip is performed ``in the presence of four (rather than six)
witnesses'', in terminology of Leclerc and Zelevinsky~\cite{LZ}, namely the
vertices $X',X'',Y',Y''$ as above.) Roughly speaking, the weak lowering flip
applied to $W=W(\tilde Y;i,j,k)$ replaces the ``middle'' vertex $X=\tilde Yik$
by the new vertex $Y$ as in~\refeq{tildeYj}, updating the tile structure in a
neighborhood of $X$. In particular, $W$ is replaced by the M-configuration
$M(\tilde Y;i,j,k)$. Weak raising flips act conversely.

In what follows, speaking of a flip, we default mean a weak flip.

Next we describe the lowering flip for $W$ in detail (using notation as above).
Define $\tilde X:=X'\cup X''$ ($=\tilde Yijk$). Two cases are possible.\medskip

\noindent\underline{\emph{Case 1}}: $\tilde X$ is a vertex of $K$. Note that
$\tilde X=X'k=Xj=X''i$. Therefore, by Proposition~\ref{pr:X-Xi}, $K$ has the
edges $e'=(X',\tilde X)$, $e=(X,\tilde X)$, and $e''=(X'',\tilde X)$. This
implies that $K$ contains the $\Delta$-tiles $\rho':=\Delta(\tilde X|X'X)$ and
$\rho'':=\Delta(\tilde X|XX'')$. We replace $\nabla',\nabla'',\rho',\rho''$ by
the triangles
  $$
  \Delta(X'|Y'Y),\quad \Delta(X''|YY''),\quad \Delta(\tilde X|X'X''),\quad
  \nabla(Y|X'X''),
  $$
as illustrated in Fig.~\ref{fig:Case1}.

\begin{figure}[htb]
\vspace{0cm}
\begin{center}
\includegraphics{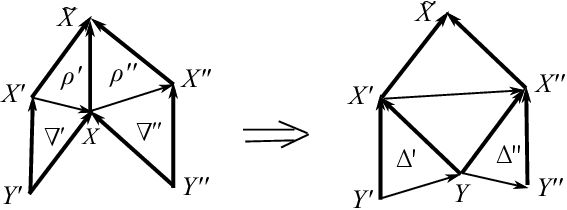}
\end{center}
\vspace{-0.5cm}
 \caption{Lowering flip in Case 1}
 \label{fig:Case1}
  \end{figure}

\noindent The tile structure of $K$ in a neighborhood of $X$ below the edges
$(Y',X)$ and $(Y'',X)$ can be of three possibilities.\medskip

\noindent\underline{\emph{Subcase 1a}}: The V-edges $(Y',X)$ and $(Y'',X)$
belong to one and the same $\Delta$-tile $\tau=\Delta(X|Y'Y'')$, and the base
$(Y',Y'')$ is shared by $\tau$ and a $\nabla$-tile $\tau'$. Since $Y'=\tilde
Yi$ and $Y''=\tilde Yk$ (cf.~\refeq{Wconf}), $\tau'$ is of the form
$\nabla(\tilde Y|Y'Y'')$. We replace $\tau,\tau'$ by the $\nabla$-tiles
  $$
  \nabla(\tilde Y|Y'Y)\quad\mbox{and}\quad \nabla(\tilde Y|YY''),
  $$
as illustrated in the left fragment of Fig.~\ref{fig:Sub1a1b}.

\begin{figure}[htb]
\vspace{0cm}
\begin{center}
\includegraphics{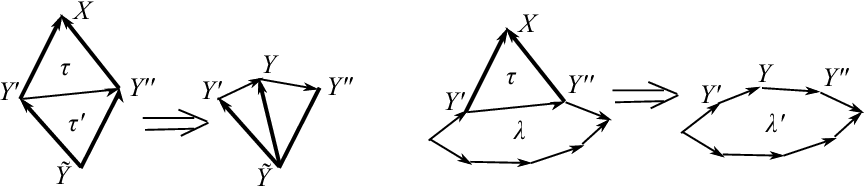}
\end{center}
\vspace{-0.5cm}
 \caption{Transformations below $X$ in Subcases~1a (left) and~1b (right)}
 \label{fig:Sub1a1b}
  \end{figure}

(Therefore, taking together, the six triangles in the hexagon with the vertices
$\tilde Y,Y',X',\tilde X,X'',Y''$ are replaced by another combination of six
triangles so as to change the inner vertex $X$ to $Y$. This matches lowering
flips in pure tilings.) \medskip

\noindent\underline{\emph{Subcase 1b}}: $(Y',X)$ and $(Y'',X)$ belong to the
same $\Delta$-tile $\tau=\Delta(X|Y'Y'')$, but the edge $e=(Y',Y'')$ is shared
by $\tau$ and a lens $\lambda$. Then $e$ belongs to the upper boundary of
$\lambda$ and the center of $U_\lambda$ is just $\tilde Y$ (since $Y'\cap
Y''=\tilde Y$). We replace the edge $e$ by the path with two edges $(Y',Y)$ and
$(Y,Y'')$, thus transforming $\lambda$ into a larger lens $\lambda'$ (which is
correct since $L_{\lambda'}=L_\lambda$, ~$|U_{\lambda'}|=|U_\lambda|+1$, and
$Y=\tilde Yj$). The transformation is illustrated in the right fragment of
Fig.~\ref{fig:Sub1a1b}.
\medskip

\noindent\underline{\emph{Subcase 1c}}: The edges $(Y',X)$ and $(Y'',X)$ belong
to different $\Delta$-tiles. Then the ``angle'' between these edges is filled
by a sequence of two or more consecutive $\Delta$-tiles
$\Delta_1=\Delta(X|Y_0Y_1)$, $\Delta_2=(X|Y_1Y_2),\ldots,
\Delta_r=(X|Y_{r-1}Y_r)$, where $r\ge 2$, $Y_0=Y'$, and $Y_r=Y''$. We replace
these triangles by one lens $\lambda$ with $U_\lambda$ formed by the path
$Y'YY''$ and with $L_\lambda$ formed by the path $Y_0Y_1\ldots Y_r$ (using path
notation via vertices). This transformation is illustrated in
Fig.~\ref{fig:Sub1c}.

\begin{figure}[htb]
\vspace{0cm}
\begin{center}
\includegraphics{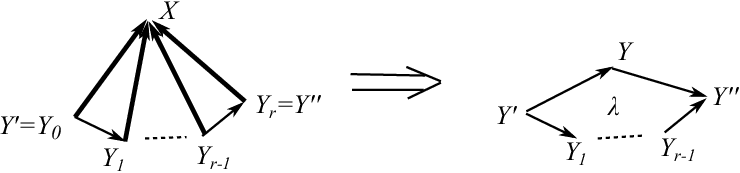}
\end{center}
\vspace{-0.5cm}
 \caption{Transformation below $X$ in Subcase~1c}
 \label{fig:Sub1c}
  \end{figure}

\noindent\underline{\emph{Case 2}}: $\tilde X$ is not a vertex of $K$. Then the
H-edges $e'=(X',X)$ and $e''=(X,X'')$ belong to the lower boundaries of some
lenses $\lambda'$ and $\lambda''$, respectively. Since $e'$ has type $jk$ and
$e''$ has type $ij$ with $i<j<k$, it follows from Proposition~\ref{pr:2edges}
that $\lambda'=\lambda''=:\lambda$. Two cases are possible. \medskip

\noindent\underline{\emph{Subcase 2a}}: $|L_{\lambda}|\ge 3$ (i.e.,
$L_{\lambda}$ contains $e',e''$ and at least one more edge). We replace
$e',e''$ by one H-edge $\tilde e=(X',X'')$, which has type $ik$. This reduces
$\lambda$ to lens $\tilde\lambda$ with $U_{\tilde\lambda}=U_{\lambda}$ and
$L_{\tilde\lambda}=(L_{\lambda}-\{e',e''\})\cup\{\tilde e\}$. (It is indeed a
lens since $|L_{\tilde\lambda}|=|L_{\lambda}|-1\ge 2$ and $X'\cup X''=\tilde
X$, thus keeping the lower center: $\Olow_{\tilde\lambda}=\Olow_{\lambda}$.)
Also, acting as in Case~1, we remove the $\nabla$-tiles $\nabla',\nabla''$ and
add the $\nabla$-tile $\nabla(Y|X'X'')$ and the two $\Delta$-tiles
$\Delta(X'|Y'Y)$ and $\Delta(X''|YY'')$. The transformation is illustrated in
the left fragment of Fig.~\ref{fig:Case2} (where $|L_{\lambda}|=4$ and
$|U_{\lambda}|=3$).

  \Xcomment{
\begin{center}
\includegraphics[scale=0.9]{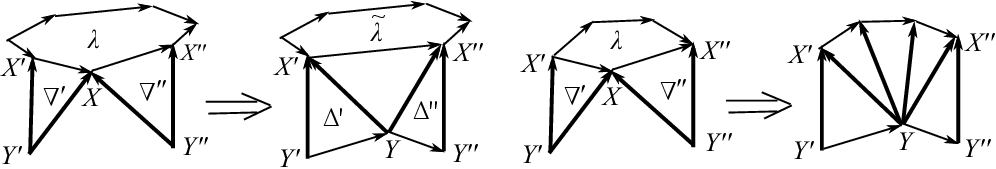}
\end{center}
\vspace{-0.3cm}
 }

\begin{figure}[htb]
\vspace{0cm}
\begin{center}
\includegraphics[scale=0.9]{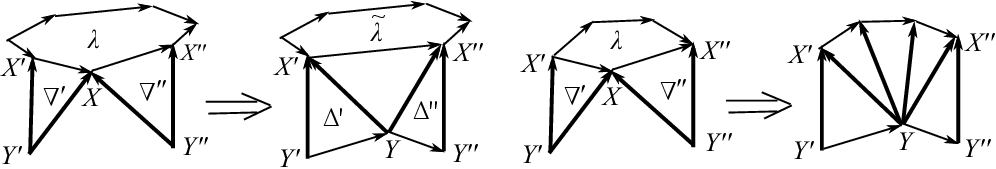}
\end{center}
\vspace{-0.5cm}
 \caption{Lowering flip in Subcases~2a (left) and~2b (right)}
 \label{fig:Case2}
  \end{figure}

\noindent\underline{\emph{Subcase 2b}}: $|L_{\lambda}|=2$. Then
$\ell_{\lambda}=X'$, $r_{\lambda}=X''$, and $L_{\lambda}$ is formed by the
edges $e',e''$. Let $X_0X_1\ldots X_q$ be the upper boundary $U_{\lambda}$,
where $X_0=X'$ and $X_q=X''$. We replace $\lambda$ by the sequence of
$\nabla$-tiles $\nabla(Y|X_0X_1), \ldots,\nabla(Y|X_{q-1}X_q)$ (taking into
account that $Y=\Oup_{\lambda}$). See the right fragment of
Fig.~\ref{fig:Case2}.
\medskip

In addition, in both Subcases 2a,2b we should perform an appropriate update of
the tile structure in a neighborhood of $X$ below the edges $(Y',X)$ and
$(Y'',X)$. This is done in the same way as described in Subcases~1a--1c.
\smallskip

A straightforward verification shows that in all cases we obtain a correct
combi $K'$. As a result of the flip, the vertex $X=\tilde Yik$ is replaced by
$Y=\tilde Yj$, and the sum of the sizes of vertices decreases by 1. The arising
$\Delta$-tiles $\Delta(X'|Y'Y)$ and $\Delta(X''|YY'')$ give an M-configuration
$M$ in $K'$, and making the corresponding raising flip with $M$, we return the
initial $K$. (The description of raising flip is ``symmetric'' to that of
lowering flip given above. In fact, a raising flip is equivalent to the
corresponding lowering flip in the combi whose vertices are the complements
$[n]-X$ of the vertices $X$ of $K$.)\smallskip

In what follows, we utilize some results from~\cite{DKK1,DKK2,LZ}.

(A) Let $\Fscr\subset 2^{[n]}$ be a maximal w-collection. Suppose that for some
triple $i<j<k$ in $[n]$ and some set $A\subseteq[n]-\{i,j,k\}$, ~$\Fscr$
contains the four sets $Ai,Ak,Aij,Ajk$ and one set $B\in\{Aj,Aik\}$ (both $Aj$
and $Aik$ cannot be simultaneously in $\Fscr$ as they are not weakly separated
from each other). Then replacing in $\Fscr$ the set $B$ by the other set from
$\{Aj,Aik\}$ (a \emph{weak flip} on set-systems) makes a maximal w-collection
as well~\cite{LZ}.

(B) Let $\Pscr=({\bf W}_n,\prec_w)$ be the poset where ${\bf W}_n$ is the set
of maximal w-collections for $[n]$ and we write $\Fscr\prec_w\Fscr'$ if
$\Fscr'$ can be obtained from $\Fscr$ by a series of (weak) raising flips. Then
$\Pscr$ has a unique minimal element and a unique maximal element, which are
the set $\Iscr_n$ of intervals and the set co-$\Iscr_n$ of co-intervals in
$[n]$, respectively~\cite{DKK1,DKK2}. \smallskip

Now we finish the proof of the theorem as follows. We associate with a combi
$K$ the parameter $\eta(K):=\sum(|X|\colon X\in V_K)$. Then a lowering
(raising) flip applied to $K$ decreases (resp. increases) $\eta$ by 1. We
assert that if $K$ admits no lowering flip, then $V_K=\Iscr_n$.

Indeed, suppose that $K$ has a (non-degenerate) lens $\lambda$. Let $\lambda$
be chosen so that $L_\lambda$ is entirely contained in the lower boundary of
the girdle $\Lambda_h$, where $h$ is the level of $\lambda$ (see
Sect.~\SSEC{combi}); one easily shows that such a $\lambda$ does exist. Take
two consecutive edges $e,e'$ in $L_\lambda$. Then $e$ has type $jk$ and $e'$
has type $ij$ for some $i<j<k$. Furthermore, $e,e'$ are the bases of some
$\nabla$-tiles $\nabla,\nabla'$, respectively (lying in level $h-\frac12$). But
then $\nabla,\nabla'$ form a W-configuration. So $K$ admits a lowering flip; a
contradiction.

Hence, $K$ is a semi-rhombus tiling. Let $T$ be its underlying rhombus tiling.
We know (see Section~\SSEC{rhomb_til}) that $T$ admits no strong lowering flip
(concerning a hexagon) if and only if $V_T=\Iscr_n$. Since $V_T=V_K$ and a
(strong) lowering flip in $T$ is translated as a (weak) lowering flip in $K$,
we obtain the desired assertion.

It follows that starting with an arbitrary combi $K$ and making a finite number
of lowering flips, we can always reach the combi $K_0$ with $V_{K_0}=\Iscr_n$
(taking into account that each application of lowering flips decreases $\eta$).
Then $K$ is obtained from $K_0$ by a series of raising flips. Each of such
flips changes the spectrum of the current combi in the same way as described
for flips concerning w-collections in~(A). Thus, $V_K$ is a maximal
w-collection.

This completes the proof of the theorem.
  \end{proof}

A converse property is valid as well.
  \begin{theorem} \label{tm:max-ws-combi}
For any maximal w-collection $\Fscr$ in $2^{[n]}$, there exists a combi $K$
such that $V_K=\Fscr$. Moreover, such a $K$ is unique.
  \end{theorem}
  \begin{proof}
~To see the existence of $K$ with $V_K=\Fscr$, it suffices to show that flips
in combies and in maximal w-collections are consistent (taking into
account~(A),(B) in the proof of Theorem~\ref{tm:combi-max-ws}). This is
provided by the following:
  \begin{numitem1} \label{eq:AAA}
for a combi $K$, if $V_K$ contains the sets $Ai,Ak,Aij,Ajk,Aik$ for some
$i<j<k$ and $A\subseteq[n]-\{i,j,k\}$, then $K$ contains the $\nabla$-tiles
$\nabla=\nabla(Ai|Aij,Aik)$ and $\nabla'=\nabla(Ak|Aik,Ajk)$.
  \end{numitem1}
(So $\nabla,\nabla'$ form a W-configuration and we can make a lowering flip in
$K$ to obtain a combi $K'$ with $V_{K'}=(V_K-\{Aik\})\cup\{Aj\}$. This matches
the flip $Aik\rightsquigarrow Aj$ in the corresponding w-collection. The
assertion on raising flips is symmetric.)

Indeed, Proposition~\ref{pr:X-Xi} ensures the existence of V-edges
$e_1=(Ai,Aij)$, $e_2=(Ai,Aik)$, $e_3=(Ak,Aik)$, $e_4=(Ak,Ajk)$ in $K$.
Obviously, the angle between the edges $e_1$ and $e_2$ is covered by one or
more $\nabla$-tiles; let $\tilde\nabla=\nabla(Ai|Aij',Aik)$ be the rightmost
tile among them (namely, the one containing the edge $e_2$). Similarly, let
$\tilde\nabla' =\nabla(Ak|Aik,Aj''k)$ be the leftmost tile among the
$\nabla$-tiles lying between the edges $e_3$ and $e_4$. Clearly $j\le j'<k$ and
$i<j''\le j$. By Proposition~\ref{pr:2edges}, $j'=j=j''$. Therefore,
$\tilde\nabla=\nabla$ and $\tilde\nabla'=\nabla'$, as required in~\refeq{AAA}.

Finally, $K$ is determined by its spectrum $V_K$ (yielding the required
uniqueness). Indeed, using Proposition~\ref{pr:X-Xi}, we can uniquely restore
the V-edges of $K$. This determines the set of $\Delta$- and $\nabla$-tiles of
$K$. Now for $h=1,\ldots,n-1$, consider the region (girdle) $\Lambda_h$ in $Z$
bounded from below by the directed path $L_h$ formed by the base edges of
$\nabla$-tiles of level $h-\frac12$, and bounded from above by the directed
path $U_h$ formed by the base edges of $\Delta$-tiles of level $h+\frac12$.
Then $\Lambda_h$ is the union of lenses of level $h$, and we have to show that
these lenses are restored uniquely.

It suffices to show this for the disk $D$ in $\Lambda_h$ between two
consecutive critical vertices $u,v$ (see Sect.~\SSEC{combi}). The lower
boundary $L_D$ of $D$ is the part of $L_h$ beginning at $u$ and ending at $v$;
let $L_D=(u=X_0,e_1,X_1,\ldots,e_r,X_r=v)$. Each edge $e_p=(X_{p-1},X_p)$
belongs to the lower boundary of some lens $\lambda$ with center
$\Olow_\lambda=X_{p-1}\cup X_p$ (and we have $X_{p-1}=\Olow_\lambda-k$ and
$X_p=\Olow_\lambda-j$ for some $j<k$). From Proposition~\ref{pr:2edges} it
follows that consecutive edges $e_p,e_{p-1}$ belong to the same lens if and
only if $X_{p-1}\cup X_p=X_p\cup X_{p+1}$. Also there exists a lens $\lambda$
such that $L_\lambda$ is entirely contained in $L_D$. Such a $\lambda$ is
characterized by the following conditions:

(i) $L_\lambda$ is a maximal part $X_pX_{p+1}\ldots X_q$ of $L_D$ satisfying
$q\ge p+2$ and $X_p\cup X_{p+1}=\cdots=X_{q-1}\cup X_q$, and

(ii) there exists a vertex $Y\ne X_p,X_q$ in $D$ such that $X_p\cap Y=Y\cap
X_q=X_p\cap X_q$.

In this case, $\ell_\lambda=X_p$, ~$r_\lambda=X_q$,
~$L_\lambda=X_pX_{p+1}\ldots X_q$, and $U_\lambda$ is formed by $X_p,X_q$ and
all those vertices $Y$ that satisfy (ii). (Conditions~(i),(ii) are necessary
for a lens $\lambda$ with $L_\lambda\subseteq L_D$. To see the sufficiency,
suppose that $L_\lambda\cap L_D=X_pX_{p+1}\ldots X_q$ but
$L_\lambda\not\subseteq L_D$, i.e., either $\ell_\lambda\ne X_p$ or
$r_\lambda\ne X_q$ or both. Consider the rounding $\Omega_\lambda$ of $\lambda$
(see Sect.~\SSEC{combi}). One can realize that a point $Y$ as in (ii) is
located in the \emph{interior} of $\Omega_\lambda$, which is impossible.)

Once a lens $\lambda$ satisfying (i),(ii) is chosen, we ``remove'' it from $D$
and repeat the procedure with the lower boundary of the updated (reduced) $D$,
and so on. Upon termination of the process for all $h$, we obtain the list of
lenses of $K$, and this list is constructed uniquely.
  \end{proof}


\section{Proofs}  \label{sec:proofs}

In this section, we prove the assertions stated but left unproved in
Sects.~\SEC{cyclic} and~\SEC{tilings}, namely
Propositions~\ref{pr:Rin-Rout},~\ref{pr:X-Xi},~\ref{pr:2edges} and
Theorem~\ref{tm:cyc_pattern}.

We will use a technique of contractions and expansions  which transform combies
on $Z_n$ into ones on $Z_{n-1}$, and back. These are analogous to those
developed for g-tilings in~\cite[Sec.~8]{DKK1} (see also~\cite[Sec.~6]{DKK2}),
but are arranged simpler because of the planarity of objects we deal with.


\subsection{$n$-contraction}  \label{ssec:contract}

The $n$-contraction operation is, in fact, well known and rather transparent
when we deal with a rhombus tiling $T$ on the zonogon $Z=Z_n$. In this case, we
take the sequence $Q$ of rhombi of type $\ast n$ (where $\ast$ means any number
$i$ different from $n$) in which the first rhombus contains the first edge
(having type $n$) in $rbd(Z)$, the last rhombus contains the last edge in $\ell
bd(Z)$, and each pair of consecutive rhombi shares an edge of type $n$. The
operation consists in shrinking each rhombus $\tau=\tau(X;i,n)$ in $Q$ into the
only edge $(X,Xi)$. Under this operation, the right boundary of $Q$ is shifted
by $-\xi_n$, getting merged with the left one, each vertex $Y$ of $T$
containing the element $n$ (i.e., $Y$ lies on the right from $Q$) turns into
$Y-n$, and the resulting set $T'$ of rhombi forms a correct rhombus tiling on
the zonogon $Z_{n-1}$ (generated by $\xi_1,\ldots,\xi_{n-1}$). This $T'$ is
called the $n$-\emph{contraction} of $T$.

In case of c-tilings, the $n$-contraction operation becomes less trivial since,
besides triangles (``semi-rhombi''), it should involve lenses of type $\ast n$.
Below we describe this in detail.

Consider a combi $K$ on $Z$. Let $\Tscr^n$ be the set of tiles of types $\ast
n$ in $K$; it consists of the $\Delta$-tiles $\Delta(A|BC)$ whose left edge
$(B,A)$ is of type $n$, the $\nabla$-tiles $\nabla(A'|B'C')$ whose right edge
$(A',C')$ is of type $n$, and the lenses $\lambda$ of type $\ast n$. Note that
the first edge of $L_\lambda$ has type $jn$, and the last edge of $U_\lambda$
has type $j'n$ for some $j,j'$.

Let $E^n$ denote the set of edges of type $n$ or $\ast n$ in $K$.

By an $n$-\emph{strip}, we mean a maximal sequence $Q=(\tau_0,\tau_1,\ldots,
\tau_N)$ of tiles in $\Tscr^n$ such that for any two consecutive
$\tau=\tau_{p-1}$ and $\tau'=\tau_p$, their intersection $\tau\cap\tau'$
consists of an edge $e\in E^n$, and:

(i) if $e$ has type $n$, then $\tau$ is a $\Delta$-tile and $\tau'$ is a
$\nabla$-tile;

(ii) if $e$ has type $\ast n$, then $\tau$ lies below $e$ and $\tau'$ lies
above $e$.

In case~(ii), there are four possibilities: (a) $\tau$ is a $\nabla$-tile,
$\tau'$ is a $\Delta$-tile, and $e$ is their common base; (b) $\tau$ is a
$\nabla$-tile, $\tau'$ is a lens, and $e$ is the base of $\tau$ and the first
edge of $L_{\tau'}$; (c) $\tau,\tau'$ are lenses, and $e$ is the last edge of
$U_\tau$ and the first edge of $L_{\tau'}$; and (d) $\tau$ is a lens, $\tau'$
is a $\Delta$-tile, and $e$ is the last edge of $U_\tau$ and the base of
$\tau'$. The possible cases of $\tau,\tau'$ are illustrated in the picture,
where their common edge $e$ is drawn bold.

\vspace{-0.3cm}
\begin{center}
\includegraphics{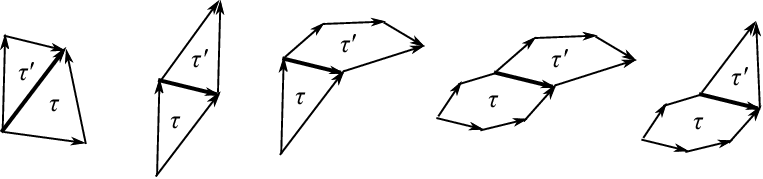}
\end{center}
\vspace{-0.3cm}

The following property is useful:
  \begin{numitem1} \label{eq:strip}
there is only one $n$-strip $Q$ as above; each tile of $\Tscr^n$ occurs in $Q$
exactly once; $Q$ begins with the $\nabla$-tile in $K$ containing the first
edge $e^r_1=(z^r_0,z^r_1)$ of $rbd(Z)$ and ends with the $\Delta$-tile in $K$
containing the last edge $e^\ell_n=(z^\ell_{n-1},z^\ell_n)$ of $lbd(Z)$.
  \end{numitem1}
This follows from three facts: (i) Each tile in $\Tscr^n$ contains exactly two
edges from $E^n$; (ii) each edge in $E^n$ belongs to exactly two tiles from
$\Tscr^n$, except for the edges $e^r_1$ and $e^\ell_n$, which belong to one
tile each; and (iii) $Q$ cannot be cyclic. To see validity of~(iii), observe
that if $\tau_{p-1},\tau_p$ are triangles, then either the level of $\tau_p$ is
greater than that of $\tau_{p-1}$ (when $\tau_{p-1}$ is a $\nabla$-tile), or
the level of the \emph{base} of $\tau_p$ is greater than that of $\tau_{p-1}$
(when $\tau_{p-1}$ is a $\Delta$-tile). As to traversing across lenses, we rely
on the fact that the relation: A lens $\lambda$ is ``higher'' than a lens
$\lambda'$ if $L_\lambda$ and $U_{\lambda'}$ share an edge, induces a poset on
the set of all lenses in $K$ (which in turn appeals to evident topological
reasonings using the facts that the lenses are convex and all H-edges are
``directed to the right'').

Note that for each tile $\tau\in\Tscr^n$ and its edges $e,e'\in E^n$,
~$bd(\tau)-\{e,e'\}$ consists of two directed paths, \emph{left} and
\emph{right} ones (of which one is a single vertex when $\tau$ is a triangle).
The concatenation of the left (resp. right) paths along $Q$ gives a directed
path from $z^r_0=(0,0)$ to $z^\ell_{n-1}$ (resp. from $z^r_1$ to
$z^\ell_n=z^r_n$); we call this path the \emph{left} (resp. \emph{right})
\emph{boundary} of the strip $Q$ and denote as $\Pleft$ (resp. $\Pright$).

Let $\Zleft$ ($\Zright$) be the region in $Z$ bounded by $\Pleft$ and the part
of $lbd(Z)$ from $(0,0)$ to $z^\ell_{n-1}$ (resp. bounded by $\Pright$ and the
part of $rbd(Z)$ from $z^r_1$ to $z^r_n$). Accordingly, the graph
$(V_K,E_K-E^n)$ consists of two connected components, the \emph{left subgraph}
$\Gleft$ and the \emph{right subgraph} $\Gright$ (lying in $\Zleft$ and
$\Zright$, respectively), and we denote by $\Kleft$ and $\Kright$ the
corresponding subtilings in $K-Q$. It is easy to see that $n\notin X$ (resp.
$n\in X$) holds for each vertex $X$ of $\Gleft$ (resp. $\Gright$).

Figure~\ref{fig:contract} illustrates the construction of $\Pleft, \Pright,
\Gleft, \Gright$ for the combi $K$ with $n=4$ from Fig.~\ref{fig:combi}; here
the edges of types $4$ and $\ast 4$ in $K$ are drawn by dotted lines.

\begin{figure}[htb]
\vspace{0cm}
\begin{center}
\includegraphics{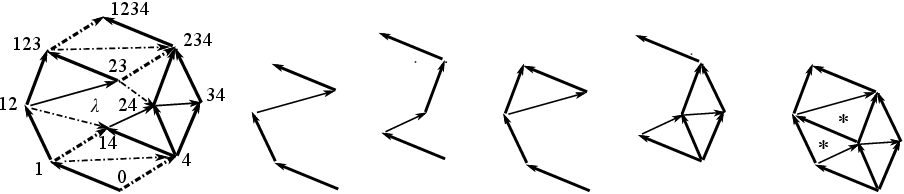}
\end{center}
\vspace{-0.5cm}
 \caption{From left to right: $K$; ~$\Pleft$; ~$\Pright$; ~$\Gleft$; ~$\Gright$; ~$K'$}
 \label{fig:contract}
  \end{figure}

The $n$-\emph{contraction operation} applied to $K$ makes the
following.\smallskip

\noindent(Q1) ~The region $\Zleft$ preserves, while $\Zright$ is shifted by the
vector $-\xi_n$. As a result, the right boundary of the shifted $\Zright$
becomes the right boundary of the zonogon $Z_{n-1}$ (while the left boundary of
$\Zleft$ coincides with $lbd(Z_{n-1})$). Accordingly, the subtiling $\Kleft$
preserves, and $\Kright$ is shifted so that each vertex $X$ becomes $X-n$.
\smallskip

\noindent(Q2) ~The edges in $E^n$ and the triangles in $\Tscr^n$ vanish.
\smallskip

\noindent(Q3) ~The lenses $\lambda$ in $\Tscr^n$ are transformed as follows.
Let $U_\lambda=X_0X_1\ldots X_p$ and $L_\lambda=Y_0Y_1\ldots Y_q$ (using path
notation via vertices); so $X_0=Y_0=\ell_\lambda$, $X_p=Y_q=r_\lambda$, and
$(Y_0,Y_1),(X_{p-1},X_p)\in E^n$. Under the above shift, each vertex $Y_d$ with
$d>0$ becomes $Y'_d:=Y_d-n$. Since $r_\lambda=\Oup_\lambda+\xi_n$ and
$Y'_q=Y_q-n$, the point $Y'_q$ coincides with $\Oup_\lambda$; hence $Y'_q$
becomes the center for $X_0,\ldots,X_{p-1}$. In its turn, the equality
$\ell_\lambda=\Olow_\lambda-\xi_n$ implies that $X_0$ becomes the center for
$Y'_1,\ldots,Y'_q$. We fill the space between the paths $X_0\ldots X_{p-1}$ and
$Y'_1\ldots Y'_q$ with new $\Delta$-tiles
$\Delta(X_0|Y'_1Y'_2),\ldots,\Delta(X_0|Y'_{q-1}Y'_q)$ and new $\nabla$-tiles
$\nabla(Y'_q|X_0X_1),\ldots,\nabla(Y'_q|X_{p-2}X_{p-1})$.

The transformation of a lens with $p=4$ and $q=3$ is illustrated in
Fig.~\ref{fig:lens_contr}.

\begin{figure}[htb]
\vspace{0cm}
\begin{center}
\includegraphics{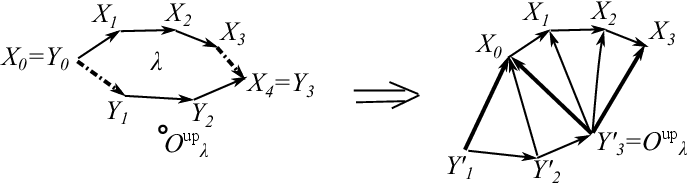}
\end{center}
\vspace{-0.5cm}
 \caption{Contraction on a lens}
 \label{fig:lens_contr}
  \end{figure}

We associate with the lens $\lambda$ as above in $K$ the zigzag path
$Y'_1X_0Y'_qX_{p-1}$, denoted as $P_\lambda$, in the resulting object (it is
drawn in bold in the right fragment of Fig.~\ref{fig:lens_contr}). One may say
that the transformation described in~(Q3) ``replaces the lens $\lambda$ by the
zigzag path $P_\lambda$ with two inscribed fillings'' (consisting of $\Delta$-
and $\nabla$-tiles), and we refer to such a transformation as a \emph{lens
reduction}. (As one more illustration, observe that the lens $\lambda$ with
$p=q=2$ from the leftmost fragment of Fig.~\ref{fig:contract} turns, under the
lens reduction, into two triangles marked by $\ast$ in the rightmost fragment
of that figure.)

Let $K'$ be the set of tiles occurring in $\Kleft$ and in the shifted $\Kright$
plus those $\Delta$- and $\nabla$-tiles that arise instead of the lenses in
$\Tscr^n$ as described in~(Q3). We call $K'$ the $n$-\emph{contraction} of $K$
and denote it as $K/n$. Analyzing the above description, one can conclude with
the following
  \begin{prop} \label{pr:n-contract}
$K/n$ is a correct combi on the zonogon $Z_{n-1}$.
  \end{prop}

Acting in a similar fashion w.r.t. the set $\Tscr^1$ of tiles of types $1\ast$
in $K$, one can construct the corresponding \emph{1-contraction} $K/1$ of $K$,
which is a combi on the $(n-1)$-zonogon generated by the vectors
$\xi_2,\ldots,\xi_n$. This corresponds to the $n$-contraction of the combi
$\hat K$ that is the mirror reflection of $K$ w.r.t. the vertical axis. (More
precisely, $\hat K$ is obtained from $K$ by changing each generator
$\xi_i=(x_i,y_i)$ to $\hat\xi_{n-i+1}=(-x_i,y_i)$.)


\subsection{$n$-expansion}  \label{ssec:expan}

Now we describe a converse operation, called the $n$-\emph{expansion one}. It
is easy when we deal with a rhombus tiling $T'$ on the zonogon $Z'=Z_{n-1}$. In
this case, we take a \emph{directed} path $P$ in $G_{T'}$ going from the bottom
$(0,0)$ to the top $z_{n-1}$ of $Z'$. Such a $P$ splits $Z'$ and $T'$ into the
corresponding left and right parts. We shift the right region of $Z'$ by the
vector $\xi_n$ (and shift the right subtiling of $T'$ by adding the element $n$
to its vertices) and then fill the space between $P$ and its shifted copy with
corresponding rhombi of type $\ast n$. This results in a corrected rhombus
tiling $T$ on the zonogon $Z_n$ in which the new rhombi form its $n$-strip, and
the $n$-contraction operation applied to $T$ returns $T'$.

In case of c-tilings, the operation becomes somewhat more involved since now we
should deal with a path $P$ in $Z'$ which is not necessarily directed.

More precisely, we consider a combi $K'$ on $Z'$ and a simple path
$P=(A_0,e_1,A_1,\ldots,$ $e_M,A_M)$ in the graph $G_{K'}$ such that:
\smallskip

\noindent(P1) ~$P$ goes from the bottom to the top of $Z'$, i.e.,
$A_0=\emptyset$ and $A_M=[n-1]$, and all edges of $P$ are V-edges; \smallskip

\noindent(P2) ~for any two consecutive edges of $P$, at least one is a forward
edge, i.e., $|A_{d-1}|>|A_d|$ implies $|A_{d}|<|A_{d+1}|$;\smallskip

\noindent(P3) ~any zigzag subpath in $P$ goes to the right; in other words, if
$|A_{d-1}|=|A_{d+1}|\ne|A_d|$, then either $A_{d-1}=A_d\,i$ and
$A_{d+1}=A_d\,j$ for some $i<j$, or $A_{d-1}=A_d-i'$ and $A_{d+1}=A_d-j'$ for
some $i'>j'$.
\smallskip

Borrowing terminology in~\cite[Sec.~8]{DKK1}, we call such a $P$ a \emph{legal}
path for $K'$. It splits $Z'$ into two closed simply connected regions
$R_1,R_2$ (the \emph{left} and \emph{right} ones, respectively), where $R_1\cup
R_2=Z'$, $R_1\cap R_2=P$, ~$R_1$ contains $lbd(Z')$, and $R_2$ contains
$rbd(Z')$. Let $G_i$ and $K_i$ denote, respectively, the subgraph of $G_{K'}$
and the subtiling of $K'$ contained in $R_i$, $i=1,2$.

We call a vertex $A_d$ of $P$ a \emph{slope} if $|A_{d-1}|<|A_d|<|A_{d+1}|$, a
\emph{peak} if $|A_{d-1}|=|A_{d+1}|<|A_d|$, and a \emph{pit} if
$|A_{d-1}|=|A_{d+1}|>|A_d|$. (By~(P2), the case $|A_{d-1}|>|A_d|>|A_{d+1}|$ is
impossible.) When $A_d$ is a peak, the angle between the edges $e_d$ and
$e_{d+1}$ is covered by a sequence of $\Delta$-tiles $\Delta(A_d|Y_{r-1}Y_r)$,
$r=1,\ldots,q$, ordered from left to right, where $Y_0=A_{d-1}$ and
$Y_q=A_{d+1}$. We call this sequence the \emph{(lower) filling} at $A_d$ w.r.t.
$P$. Similarly, when $A_d$ is a pit, the angle between $e_d$ and $e_{d+1}$ is
covered by a sequence of $\nabla$-tiles $\nabla(A_d|X_{r-1}X_r)$,
$r=1,\ldots,p$, called the \emph{(upper) filling} at $A_d$ w.r.t. $P$.

The $n$-\emph{expansion operation} applied to $K'$ and $P$ constructs a combi
$K$ on $Z_n$ as follows. \smallskip

\noindent(E1) ~$K$ inherits all tiles of $K_1$ except for those in the fillings
of pits of $P$. For each tile $\tau$ of $K_2$ not contained in the filling of
any peak of $P$, $K$ receives the shifted tile $\tau+\xi_n$. Accordingly, the
vertex set of $K$ consists of the vertices of $G_1$ except for the pits of $P$,
and the vertices of the form $Xn$ for all vertices $X$ of $G_2$ except for the
peaks of $P$. In particular, each slope $X$ of $P$ (and only these vertices of
$K'$) creates two vertices in $K$, namely $X$ and $Xn$.\smallskip

\noindent(E2) ~Each slope $A_d$ creates two additional tiles in $K$: the
$\nabla$-tile $\nabla(A_d|A_{d+1},A_d\,n)$ of type $in$ and the $\Delta$-tile
$\Delta(A_d\,n|A_d,A_{d-1}\,n)$ of type $jn$, where $i,j$ are the types of the
edges $(A_d,A_{d+1})$ and $(A_{d-1},A_d)$, respectively. See the left fragment
of Fig.~\ref{fig:expan}.

\begin{figure}[htb]
\vspace{0cm}
\begin{center}
\includegraphics[scale=0.95]{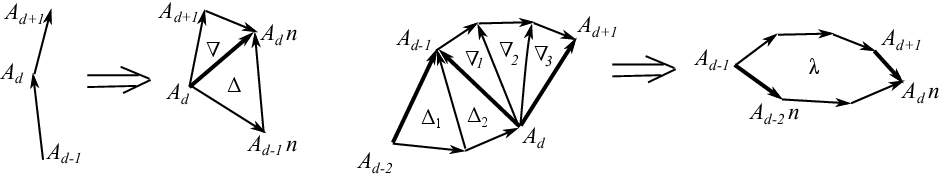}
\end{center}
\vspace{-0.5cm}
 \caption{$n$-expansion in a neighborhood of a slope (left) and a zigzag (right)}
 \label{fig:expan}
  \end{figure}

\noindent In addition, the first and last vertices of $P$ create two extra
tilings: $\nabla(A_0=\emptyset|A_1,A_0\,n=z^r_1)$ and
$\Delta(A_M\,n=[n]|A_M,A_{M-1}\,n)$. \smallskip

\noindent(E3) ~Each backward edge $e_d=(A_d,A_{d-1})$ of $P$ (equivalently,
each peak-pit pair $A_{d-1},A_d$) creates a new lens $\lambda$ in $K$ by the
following rule. Let $\Delta(A_{d-1}|Y_{r-1}Y_r)$, $r=1,\ldots,q$, be the
filling at $A_{d-1}$ w.r.t. $P$, and let $\nabla(A_d|X_{r-1}X_r)$,
$r=1,\ldots,p$, be the filling at $A_d$ (so $Y_0=A_{d-2}$, ~$Y_q=A_d$,
~$X_0=A_{d-1}$, and $X_p=A_{d+1}$). Then $\lambda$ is such that:
$\ell_\lambda=A_{d-1}$, $r_\lambda=A_d\,n$, the upper boundary $U_\lambda$ has
the vertex sequence $X_0,X_1,\ldots,X_p,A_d\,n$, and the lower boundary
$L_\lambda$ the sequence $A_{d-1},Y_0\,n,Y_1\,n,\ldots,Y_q\,n$. See the right
fragment of Fig.~\ref{fig:expan} (where $p=3$ and $q=2$). We refer to the
transformation in~(E3) as a \emph{lens creation}; this is converse to a lens
reduction described in~(Q3) of Sect.~\SSEC{contract}.
\smallskip

One can check that $K$ obtained in this way is indeed a correct combi on $Z_n$.
It is called the $n$-\emph{expansion} of $K'$ w.r.t. the legal path $P$. The
corresponding sequence $Q$ of tiles of types $\ast n$ in $K$ is just formed by
the $\Delta$- and $\nabla$-tiles induced by the slopes of $P$ (together with
$A_0,A_M$) as described in~(E2), and by the lenses induced by the backward
edges (or the three-edge zigzags) of $P$ as described in~(E3).

A straightforward examination shows that the $n$-contraction operation applied
to $K$ returns the initial $K'$ (and the legal path $P$ in it is obtained by a
natural deformation of the left boundary of $Q$); details are left to the
reader. As a result, we conclude with the following.
  \begin{theorem} \label{tm:contr-expand}
The correspondence $(K',P)\mapsto K$, where $K'$ is a c-tiling on $Z_{n-1}$,
~$P$ is a legal path for $K'$, and $K$ is the $n$-expansion of $K'$ w.r.t. $P$,
gives a bijection between the set of such pairs $(K',P)$ and the set of
c-tilings on $Z_n$. Under this correspondence, $K'$ is the $n$-contraction
$K/n$ of $K$.
  \end{theorem}

Note that when we consider in $G_{K'}$ a path $P'$ defined similarly to $P$
with the only difference (in~(P3)) that any zigzag subpath in $P'$ goes
\emph{to the left}, then, duly modifying the expansion operation described
in~(E1)--(E3), we obtain what is called the \emph{1-expansion} of $K'$; this
gives an analog of Theorem~\ref{tm:contr-expand} concerning type 1. (Again, to
clarify the construction we can make the mirror reflection w.r.t. the vertical
axis.)


\subsection{Proofs}  \label{ssec:proofs}

We first prove Propositions~\ref{pr:X-Xi} and~\ref{pr:2edges}, thus completing
the proof of Theorem~\ref{tm:combi-max-ws}.\medskip

\noindent\emph{Proof of Proposition~\ref{pr:X-Xi}}. ~Let $X$ and $Xi$ be
vertices of a combi $K$ on the zonogon $Z_n$. We use induction on $n$ and,
assuming w.l.o.g. that $i\ne n$, consider the $n$-contraction $K'=K/n$ of $K$
(in case $i=n$, we should consider the 1-contraction $K/1$ and argue in a
similar way). Also we use terminology, notation and constructions from the
previous subsections. Two cases are possible.

(i) Let $n\notin X$. Then $X$ and $Xi$ belong to the left subgraph $\Gleft$ of
$G_K$ (w.r.t. the $n$-strip $Q$). Hence $X,Xi$ are vertices of $K'$, and by
induction $K'$ has edge $e=(X,Xi)$ (which is a V-edge). Let $P$ be the legal
path in $K'$ such that the $n$-expansion of $K'$ w.r.t. $P$ gives $K$ (i.e.,
$P$ is the ``image'' of the strip $Q$ in $K'$). Then $e$ is a V-edge of the
left subgraph $G_1$ of $G_{K'}$ (w.r.t. $P$). The only situation when $e$ might
be destroyed under the $n$-expansion operation is that $e$ belongs to a
$\nabla$-tile in the filling at some pit $A_d$ of $P$, implying $X=A_d$. But
this is impossible since $A_d$ induces only one vertex $A_d\,n$ in $K$.

(ii) Let $n\in X$. Then $X,Xi$ belong to $\Gright$. So $K'$ has vertices
$X'=X-n$ and $X''=Xi-n$, and by induction $K'$ has V-edge $e=(X',X'')$. For the
corresponding legal path $P$ in $K'$, the edge $e$ could not generate the edge
$(X'n,X''n)=(X,Xi)$ in $K$ only if $e$ belongs to a $\Delta$-tile in the
filling at some peak $A_d$ of $P$, implying $X''=A_d$. But $A_d$ induces only
one vertex $A_d$ in $K$ (instead of the required vertex $A_d\,n=Xi$); a
contradiction.

Thus, in all cases, $(X,Xi)$ is an edge of $K$, and we are done. \hfill\qed
\medskip

\noindent\emph{Proof of Proposition~\ref{pr:2edges}}. ~Let $e=(A,B)$,
$e'=(B,C)$, and $i<j''\le j'<k$ be as in the hypotheses of the first statement
in this proposition. Consider two cases. \medskip

\noindent\underline{\emph{Case 1}}. Suppose that the vertex $B$ has at least
one outgoing V-edge. Let $\tilde e=(B,D)$ and $\tilde e'=(B,D')$ be the
leftmost and rightmost edges among such V-edges, respectively. Then there exist
$\Delta$-tiles of the form $\Delta:=\Delta(D|A'B)$ and
$\Delta':=\Delta(D'|BC')$ (i.e., lying on the left of $\tilde e$ and on the
right of $\tilde e'$, respectively). Let the V-edges
$(A',D),(B,D),(B,D'),(C',D')$ be of types $\tilde k, \tilde j',\tilde
j'',\tilde i$, respectively. By the choice of $\tilde e$ and $\tilde e'$, we
have $\tilde j'\le\tilde j''$.

Next, the base $(A',B)$ of $\Delta$ has type $\tilde j'\tilde k$ and, by the
planarity of $K$, should lie above the edge $(A,B)$ (admitting the equality
$(A',B)=(A,B)$). This implies $j'\le \tilde j'$ (and $k\le \tilde k$), by
comparing $\Delta$ with the abstract $\Delta$-tile with the base $(A,B)$,
which, obviously, has type $j'k$. For a similar reason, since the base $(B,C')$
of $\Delta'$ lies above $(B,C)$, we have $\tilde j''\le j''$. Therefore,
$\tilde j''\le j''\le j'\le\tilde j'\le \tilde j''$. This gives equality
throughout, implying that $\Delta=\Delta(D|AB)$ and $\Delta'=\Delta(D|BC)$, as
required. \medskip

\noindent\underline{\emph{Case 2}}. Let $B$ have no outgoing V-edges. Then $B$
is an intermediate vertex in the lower boundary of some lens $\lambda$. Let
$\tilde e=(A',B)$ and $\tilde e'=(B,C')$ be the edges of $L_\lambda$ entering
and leaving $B$, respectively. Then $\tilde e,\tilde e'$ have types $\tilde
j\tilde k$ and $\tilde i\tilde j$ for some $\tilde i<\tilde j<\tilde k$. Since
the edge $\tilde e$ lies above $e$, and $\tilde e'$ lies above $e'$ (admitting
equalities), we have $j'\le\tilde j\le j''$. This together with $j''\le j'$
gives $e=\tilde e$ and $e'=\tilde e'$, and the result follows.

The second assertion in the proposition is symmetric. \hfill\qed

This completes the proof of Theorem~\ref{tm:combi-max-ws}.
\medskip

Now we return to a simple cyclic pattern $\Sscr=(S_1,\ldots,S_r=S_0)$, i.e.,
such that $|S_{i-1}\triangle S_i|=1$ for all $i$, and the sets $S_i$ are
different and weakly separated from each other (conditions (C1),(C2) in
Sect.~\SEC{cyclic}). By Theorem~\ref{tm:max-ws-combi}, $\Sscr$ is included in
the spectrum (vertex set) $V_K$ of some combi $K$ on $Z_n$.
Proposition~\ref{pr:X-Xi} implies that each pair of consecutive vertices
$S_{i-1},S_i$ is connected in $G_K$ by a V-edge directed from the smaller set
to the bigger one. This gives the corresponding cycle in $G_K$; we identify it
with the curve $\zeta_\Sscr$ in $Z_n$ (defined in Sect.~\SEC{cyclic}). The
planarity of $G_K$ implies that $\zeta_\Sscr$ is non-self-intersecting, as
required in Proposition~\ref{pr:Rin-Rout}.

Next we finish the proof of Theorem~\ref{tm:cyc_pattern}, as follows. Consider
arbitrary $X\in\Din$ and $Y\in\Dout$. Since $\Sscr\cup\{X\}$ is weakly
separated, it is included in the spectrum $V_K$ of some combi $K$, by
Theorem~\ref{tm:max-ws-combi}. Let $\Kin$ be the subtiling of $K$ formed by the
tiles occurring in $\Rin$. Then $X$ is a vertex of $\Kin$. In its turn,
$\Sscr\cup\{Y\}\subseteq V_{K'}$ for some combi $K'$, and $Y$ is a vertex of
the subtiling $\Kpout$ of $K'$ formed by the tiles occurring in $\Rout$ (or $Y$
is a vertex in $bd(Z_n)$). Since $\Rin\cap\Rout=\zeta_\Sscr$, the union
$\Kin\cup \Kpout$ gives a correct combi $\tilde K$ on $Z_n$ containing both $X$
and $Y$. Since $V_{\tilde K}$ is a w-collection (by
Theorem~\ref{tm:combi-max-ws}), we have $X\weak Y$, as required. Therefore,
$\Din,\Dout$ form a complementary pair (their union $\Dscr$ is w-pure). This
implies the w-purity of $\Din$ and $\Dout$, by Proposition~\ref{pr:complement}.
\hfill\qed
\medskip

The strong separation counterpart of this theorem, namely
Theorem~\ref{tm:strong_cycl}, is obtained in a similar way. Given a cyclic
pattern $\Sscr$ consisting of different strongly separated sets, the result
immediately follows from Theorem~\ref{tm:LZ_strong} and the fact that any
rhombus tiling on $\Rin$ and any rhombus tiling on $\Rout$ can be combined to
form a rhombus tiling on $Z_n$. The equality $r^s(\Dscr)=r^w(\Dscr)$ is
obvious, where $\Dscr=\hat \Dscr_{\Sscr}^{\rm in}$ or $\hat\Dscr_{\Sscr}^{\rm
out}$.
\medskip

\noindent\textbf{Remark 3.} Using combies, one can essentially simplify the
proof of the fundamental property that $2^{[n]}$ is w-pure, which is given
in~\cite{DKK2} by use of generalized tilings. On this way, we first weaken the
statement of Theorem~\ref{tm:combi-max-ws} as follows: for a combi $K$, the
spectrum $V_K$ is a \emph{largest} w-collection (i.e., it has the maximal
possible size $n(n+1)/2+1$). This is proved just as in Sect.~\ref{ssec:flips}
but without appealing to results in~\cite{DKK2}. Second, as is explained
in~\cite[Sec.~4]{DKK2}, the assertion that any maximal w-collection is largest
is reduced to the following: \emph{the partial order $\prec^\ast$ on a largest
w-collection $\Fscr$, given by $A\prec^\ast B\Leftrightarrow (A\lessdot B\;
\&\; |A|\le |B|)$, forms a lattice}, where $\lessdot$ is defined
in~\refeq{2relat}(iii). The core of the whole proof consists in showing that
the partial order $(\Fscr,\prec^\ast)$ coincides with the natural partial order
on the vertices of the combi $K$ associated with $\Fscr$ (where a vertex $u$ is
less than $v$ if there is a directed path from $u$ to $v$ in $G_K$); cf.
Theorem~6.1 in~\cite{DKK2}. This is proved by induction on $n$, using a
technique of $n$-contractions and $n$-expansions on combies (which is simpler
than an analogous technique for g-tilings). We omit details here.


\section{Special cases and generalizations}  \label{sec:gen}

In this section, we discuss possible ways to extend the w-purity result to more
general patterns. Along the way we demonstrate some representative special
cases.


\subsection{Semi-simple cyclic patterns}  \label{ssec:semi-simple}

We can slightly generalize Theorem~\ref{tm:cyc_pattern} by weakening
condition~(C1) in Section~\SEC{cyclic}. Now we admit, with a due care, cyclic
patterns $\Sscr$ having some repeated elements $C_i=C_j$. More precisely, in
the corresponding closed curve $\zeta_\Sscr$ in $Z=Z_n$ we allow only touchings
but not crossings, which is equivalent to saying that under a ``very small''
deformation the curve becomes non-self-intersecting. We refer to $\Sscr$
satisfying this requirement and condition~(C2) as \emph{semi-simple}. The
definitions of regions $\Rin,\Rout$ and domains $\Din,\Dout$ are modified in a
natural way. The desired generalization reads as follows:
  \begin{itemize}
\item[($\ast$)] \emph{For a semi-simple cyclic pattern $\Sscr$, the domains $\Din$ and
$\Dout$ form a complementary pair; as a consequence, both $\Din$ and $\Dout$
are w-pure.}
  \end{itemize}
This is shown in a similar way as Theorem~\ref{tm:cyc_pattern} and we omit a
proof here.
\smallskip

One can see that domains of types $\Dscr(\omega)$ and $\Dscr(\omega',\omega)$
exposed in Theorem~\ref{tm:LZconj}(i),(ii) are representable as $\Din$ for
special simple or semi-simple cyclic patterns $\Sscr$.

In particular, for the domain $\Dscr=\Dscr(\omega',\omega)$ with permutations
$\omega',\omega:[n]\to[n]$ such that $\Inv(\omega')\subset \Inv(\omega)$, the
cyclic pattern $\Sscr$ involves $\emptyset$, $[n]$, and the sets
${\omega'}^{-1}([i])$ and $\omega^{-1}([i])$ for $i=1,\ldots,n-1$; such an
$\Sscr$ is simple if ${\omega'}^{-1}([i])\ne\omega^{-1}([i])$ for all $i$, and
semi-simple otherwise. The picture below illustrates $\Sscr$ for
$\Dscr(\omega',\omega)$ with $n=4$ in the case $\omega'=1324$ and $\omega=3241$
(left), and in the case $\omega'=1324$ and $\omega=3142$ (right).

\begin{center}
\includegraphics{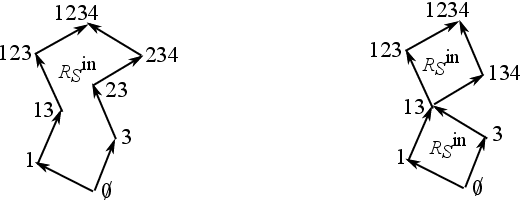}
\end{center}
\vspace{-0.3cm}


\subsection{Generalized cyclic patterns} \label{ssec:gen_cyc_pat}

Next we discuss another, more important way to extend the obtained w-purity
results, namely, we consider a \emph{generalized} cyclic pattern
$\Sscr=(S_1,\ldots,S_r=S_0)$ in $2^{[n]}$ (defined in Sect.~1). As before,
$\Sscr$ obeys~(C1) and~(C2), and now we assume that
  \begin{numitem1} \label{eq:genSscr}
for each $p=1,\ldots, r$, either $|S_{p-1}\triangle S_p|=1$, or
$|S_{p-1}\triangle S_p|=2$ and $|S_{p-1}|=|S_p|$.
  \end{numitem1}
In the former (latter) case, we say that $\{S_{p-1},S_p\}$ is a
\emph{1-distance pair} (resp. \emph{2-distance pair}) in $\Sscr$. In the latter
case, $S_{p-1}-S_p$ and $S_p-S_{p-1}$ consist of some singletons $i$ and $j$,
respectively, and setting $X:=S_{p-1}\cap S_p$ and $Y:=S_{p-1}\cup S_p$, we
have
  $$
  S_{p-1}=Xi=Y-j\quad \mbox{and}\quad S_p=Xj=Y-i.
  $$

We rely on the following assertion. (A property of a somewhat similar flavor
for plabic graphs is established in~\cite{OPS}.)
  \begin{prop} \label{pr:XiXj}
Suppose that a combi $K$ has two vertices of the form $A=Xi$ and $B=Xj$. Then
at least one of the following takes place:

{\rm(i)} $K$ contains the vertex $X$ (and therefore the edges $(X,A)$ and
$(X,B)$);

{\rm(ii)} $K$ contains the vertex $Xij$ (and therefore the edges $(A,Xij)$ and
$(B,Xij)$);

{\rm(iii)} both $A,B$ belong to the lower boundary of some lens in $K$;

{\rm(iv)} both $A,B$ belong to the upper boundary of some lens in $K$.
   \end{prop}
   \begin{proof}
Let for definiteness $i<j$. Like the proof of Proposition~\ref{pr:X-Xi}, we use
induction on $n$ and consider the $n$-contraction $K'$ of $K$. Then $K$ is the
$n$-expansion of $K'$ w.r.t. a legal path $P$ in $G_{K'}$. We first assume that
$j<n$ and consider two possible cases. \medskip

\noindent\underline{\emph{Case 1}}. ~Let $n\notin X$. Then $A,B$ are vertices
of $K'$ contained in the left subgraph $G_1$ of $G_{K'}$ (w.r.t. $P$). By
induction either (a) both $A,B$ belong to one boundary ($L_\lambda$ or
$U_\lambda$) of some lens $\lambda$ in $K'$, or (b) $K'$ contains $X=A\cap B$,
or (c) $K'$ contains $Xij=A\cup B$. In case~(a), $\lambda$ continues to be a
lens in $K$, yielding (iii) or (iv) in the proposition. In case~(b), if $X$ is
not a pit of $P$, then $X$ is a vertex of $K$, as required in~(i). And if $X$
is a pit of $P$, then $A,B$ belong to $\nabla$-tiles in the filling at $X$
w.r.t. $P$, and therefore $A,B$ become vertices of $U_\lambda$ for the lens
$\lambda$ arising in $K$ in place of the three-edge zigzag in $P$ containing
$X$; this yields (iv). In case~(c), $Xij$ cannot be a pit of $P$ (otherwise
$A,B$ would be not in $G_1$). Hence $Xij$ is a vertex of $K$, as required
in~(ii).
\medskip

\noindent\underline{\emph{Case 2}}. ~Let $n\in X$. Then $A':=A-n$ and $B':=B-n$
are vertices of $K'$ contained in the right subgraph $G_2$ of $G_{K'}$. By
induction either (a$'$) both $A',B'$ belong to one of $L_\lambda$, $U_\lambda$
for some lens $\lambda$ of $K'$, or (b$'$) $K'$ contains $A'\cap B'$, or (c$'$)
$K'$ contains $A'\cup B'$. Cases~(a$'$) and~(b$'$) are easy, as well us~(c$'$)
unless $A'\cup B'$ is a peak of $P$. And if $A'\cup B'$ is a peak of $P$, then
$A',B'$ belong to $\Delta$-tiles in the filling at $A'\cup B'$ w.r.t. $P$,
implying that $A=A'n$ and $B=B'n$ are vertices of $L_\lambda$ for the
corresponding lens $\lambda$ of $K$. \medskip

Now assume that $j=n$. Then $n\notin X$ and $n\notin A$. Hence $A$ is a vertex
of the subgraph $G_1$, and $B\cap[n-1]=X$ is a vertex of $G_2$ of $K'$. Also
$A=Xi$, and by Proposition~\ref{pr:X-Xi}, $K'$ has the V-edge $(X,A)$ of type
$i$. Consider two cases. \medskip

\noindent\underline{\emph{Case 3a}}. ~$X$ is a slope or a peak of $P$. Then $X$
is a vertex of $K$, yielding~(i) in the proposition. \medskip

\noindent\underline{\emph{Case 3b}}. ~$X$ is a pit of $P$. Then $A$ belongs to
a $\nabla$-tile in the filling at $X$ w.r.t. $P$, whence $A$ is a vertex of
$U_\lambda$ for the corresponding lens $\lambda$ of $K$. The right vertex
$r_\lambda$ of $\lambda$ is just of the form $Xn=B$, and we obtain~(iv) in the
proposition.
   \end{proof}

Like simple cyclic patterns in Sect.~\SEC{cyclic}, the sets $S_p$ are
identified with the corresponding points in the zonogon $Z=Z_n$, and we connect
each pair $S_{p-1},S_p$ by line segment $e_p$, obtaining the closed piecewise
linear curve $\zeta_\Sscr$ in $Z$. We direct each $e_p$ so as to be congruent
to the corresponding generator $\xi_i$ or vector $\eps_{ij}$.

A reasonable question arises: When $\zeta_\Sscr$ is non-self-intersecting?
Proposition~\ref{pr:XiXj} enables us to find necessary and sufficient
conditions in terms of ``forbidden quadruples'' in $\Sscr$. These conditions
are as follows (cf. Proposition~\ref{pr:non-self-int}).\smallskip

\noindent\textbf{(C3)} ~$S$ contains no quadruple $S_{p-1},S_p,S_{q-1},S_q$
such that either $\{S_{p-1},S_p\}=\{Xi,Xk\}$ and
$\{S_{q-1},S_q\}=\{Xj,X\ell\}$, or $\{S_{p-1},S_p\}=\{X-i,X-k\}$ and
$\{S_{q-1},S_q\}=\{X-j,X-\ell\}$, where $i<j<k<\ell$. \smallskip

\noindent\textbf{(C4)} ~$S$ contains no quadruple $S_{p-1},S_p,S_{q-1},S_q$
such that either $\{S_{p-1},S_p\}=\{Xi,Xk\}$ and $\{S_{q-1},S_q\}=\{X,Xj\}$, or
$\{S_{p-1},S_p\}=\{X-i,X-k\}$ and $\{S_{q-1},S_q\}=\{X,X-j\}$, where $i<j<k$.
 \smallskip

To prove the assertion below and for further purposes, we need to refine some
definitions. Fix a combi $K$ and consider a vertex $A$ in it. By the
\emph{(full) upper filling} at $A$ we mean the sequence
$\nabla(A|X_0X_1),\ldots, \nabla(A|X_{q-1}X_q)$ of all $\nabla$-tiles having
the bottom $A$ and ordered from left to right (i.e., $X_0X_1\ldots X_q$ is a
directed path in $G_K$). The union of these tiles is called the \emph{upper
sector} at $A$ and denoted by $\Sigmaup_A$, and the path $X_0X_1\ldots X_q$ is
called the \emph{upper boundary} of $\Sigmaup_A$ and denoted by $U_A$.
Symmetrically, the \emph{(full) lower filling} at $A$ is the sequence
$\Delta(A|Y_0Y_1),\ldots, \Delta(A|Y_{q'-1}Y_{q'})$ of all $\Delta$-tiles
having the top $A$ and ordered from left to right, the \emph{lower sector}
$\Sigmalow_A$ at $A$ is the union of these tiles, and the \emph{lower boundary}
$L_A$ of $\Sigmalow_A$ is the directed path $Y_0Y_1\ldots Y_{q'}$. Note that
one of these fillings or both may be empty.
  \begin{prop} \label{pr:non-self-int}
For a generalized cyclic pattern $\Sscr$, the curve $\zeta_\Sscr$ is
non-self-intersecting if and only if $\Sscr$ satisfies (C1)--(C4).
  \end{prop}
 \begin{proof}
It is easy to see that if $\Sscr$ has a quadruple as in~(C3) or (C4), then the
corresponding segments (edges) $e_p$ and $e_q$ are crossing (have a common
interior point), and therefore $\zeta_\Sscr$ is self-intersecting.

Conversely, suppose that $\zeta_\Sscr$ is self-intersecting. To show the
existence of a pair as in~(C3) or~(C4), fix a combi $K$ with $V_K$ including
$\Sscr$. Since all sets in $\Sscr$ are different, $\zeta_\Sscr$ has two
crossing segments $e_p$ and $e_q$. The edges of $K$ are non-crossing (since $K$
is planar); therefore, at least one of $e_p, e_q$ is not an edge of $K$. Let
for definiteness $e_p$ be such, i.e., $\{S_{p-1},S_p\}$ is a 2-distance pair.

By Proposition~\ref{pr:XiXj}, at least one of the following takes place: (i)
$X:=S_{p-1}\cap S_p\in V_K$; (ii) $Y:=S_{p-1}\cup S_p\in V_K$; (iii) both
$S_{p-1},S_p$ belong to the same boundary, either $U_\lambda$ or $L_{\lambda}$,
for some lens $\lambda$ of $K$. In case (iii), the only possibility for $e_q$
to cross $e_p$ is when $\{S_{q-1},S_q\}$ is a 2-distance pair occurring in the
same boundary (either $U_\lambda$ or $L_\lambda$) of $\lambda$ where
$S_{p-1},S_p$ are contained; moreover, the elements of these two pairs should
be intermixing in this boundary. This gives a quadruple as in~(C3).

In case~(i), both vertices $S_{p-1},S_p$ (being of the form $Xi,Xk$ for some
$i,k$) lie in the boundary $U_X$ of the upper sector $\Sigmaup_X$ at $X$. Then
$e_q$ can cross $e_p$ only in two cases: (a) the pair $\{S_{q-1},S_q\}$ lies in
$U_X$ as well and, moreover, its elements and those of $\{S_{p-1},S_p\}$ are
intermixing in $U_X$ (yielding a quadruple as in~(C3)); and (b) one of
$S_{q-1},S_q$ is just $X$ while the other belongs to $U_X$ and, moreover, the
latter lies between $S_{p-1}$ and $S_p$ (yielding a quadruple as in~(C4)). The
case~(ii) is symmetric to~(i) and we argue in a similar way.
  \end{proof}

To extend Theorem~\ref{tm:cyc_pattern} to a generalized cyclic pattern
$\Sscr=(S_1,\ldots,S_r)$, we will consider one or another combi $K$ with
$\Sscr\subseteq V_K$. Unlike the case of simple cyclic patterns, it now becomes
less trivial to split $K$ into two subtilings $\Kin$ and $\Kout$ (lying in the
regions $\Rin$ and $\Rout$, respectively). A trouble is that some 2-segment
$e_p=[S_{p-1},S_p]$ may cut some tile $\tau$ of $K$ (i.e., $e_p$ and $\tau$
have an interior point in common), in contrast to 1-segments, which correspond
to V-edges and therefore cannot cut any tile of $K$. Hereinafter we refer to
the line segment $[S_{p-1},S_p]$ connecting points $S_{p-1},S_p$ in the zonogon
as a \emph{1-segment} (resp. \emph{2-segment}) if these points form a
1-distance (resp. 2-distance) pair.

We overcome this trouble by use of the \emph{splitting method} described below.
It works somewhat differently for lenses and for triangles. We use an important
fact that can be deduced from Proposition~\ref{pr:XiXj}: For a 2-segment
$[S_{p-1},S_p]$ cutting a tile $\tau$ of $K$, if $\tau$ is a lens, then both
vertices $S_{p-1},S_p$ belong to the boundary either $U_\tau$ or $L_\tau$; if
$\tau$ is a $\Delta$-tile $\Delta(A|BC)$, then $S_{p-1},S_p$ belong to the
lower boundary of the sector $\Sigma_A^{{\rm low}}$; and if $\tau$ is a
$\nabla$-tile $\nabla(A|BC)$, then $S_{p-1},S_p$ belong to the upper boundary
of $\Sigma_A^{{\rm up}}$.
\smallskip

\noindent \textbf{I.} ~First we consider a lens $\lambda$ of $K$ such that
$\zeta_\Sscr$ cuts $\lambda$ (otherwise there is no problem with $\lambda$ at
all). The curve $\zeta_\Sscr$ may go across $\lambda$ several times; let
$e_{p(1)},\ldots, e_{p(d)}$ be the 2-segments cutting $\lambda$. These segments
are pairwise non-crossing (by~(C3)) and subdivide $\lambda$ into $d+1$ polygons
$D_1,\ldots,D_{d+1}$; so each $e_{p(i)}$ is $D_j\cap D_{j'}$ for some $j,j'$
(and one of $D_j,D_{j'}$ lies in $\Rin$, and the other in $\Rout$).

If $\{S_{p-1},S_p\}=\{\ell_\lambda,r_\lambda\}$, we say that $e_{p(i)}$ is the
\emph{central segment}. Otherwise, both ends of $e_{p(i)}$ belong to either
$U_\lambda$ or $L_\lambda$; in the former (latter) case, $e_{p(i)}$ is called
an \emph{upper} (resp. \emph{lower}) \emph{segment}, and we associate to it the
path $P_{p(i)}$ in $U_\lambda$ (resp. $L_\lambda$) connecting $S_{p(i)-1}$ and
$S_{p(i)}$. In view of~(C3), such paths form a \emph{nested family}, i.e., for
any $i\ne i'$, either the interiors of $P_{p(i)}$ and $P_{p(i')}$ are disjoint,
or $P_{p(i)}\subset P_{p(i')}$, or $P_{p(i)}\supset P_{p(i')}$.

Accordingly, polygons $D_j$ can be of three sorts. When $D_j$ has all vertices
in $U_\lambda$, its lower boundary is formed by exactly one (upper or central)
segment (while its upper boundary is formed by some edges of $P_{p(i)}$ and
segments $e_{p(i')}$); we call such a $D_j$ an \emph{upper semi-lens}.
Symmetrically, when $D_j$ has all vertices in $L_\lambda$, its upper boundary
is formed by exactly one (lower or central) segment, and we call $D_j$ a
\emph{lower semi-lens}. Besides, when the central segment does not exist, there
appears one more polygon $D_j$; it is viewed as an (abstract) lens $\lambda'$
with $\ell_{\lambda'}=\ell_\lambda$ and $r_{\lambda'}=r_\lambda$, and we call
it (when exists) a \emph{secondary} lens.

A possible splitting of a lens $\lambda$ is illustrated in
Fig.~\ref{fig:split_lens}, where $d=4$ and the four cutting segments are
indicated by dotted lines.

\begin{figure}[htb]
\vspace{0cm}
\begin{center}
\includegraphics{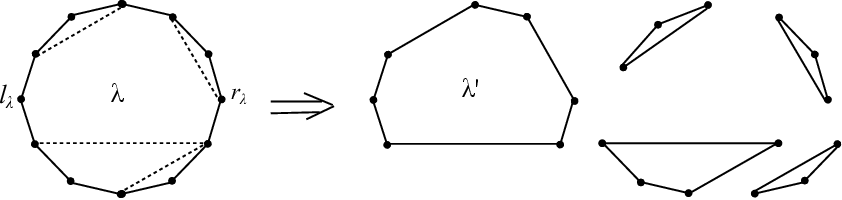}
\end{center}
\vspace{-0.5cm}
 \caption{Splitting a lens}
 \label{fig:split_lens}
  \end{figure}

\noindent \textbf{II.} ~Next suppose that $\zeta_\Sscr$ cuts some $\Delta$-tile
$\Delta(A|BC)$ of $K$. Then the lower sector $\Sigmalow_A$ at $A$ is cut by
some 2-segments in $\zeta_\Sscr$; let $e_{p(1)},\ldots,e_{p(d)}$ be these
2-segments. (Besides, $\Sigmalow_A$ can be cut by some 1-segments.) For each
segment $e_{p(i)}$, its ends $S_{p(i)-1}$ and $S_{p(i)}$ belong to the lower
boundary $L_A$ of the sector, and we denote by $P_{p(i)}$ the directed path in
$L_A$ connecting these vertices. Such paths form a nested family, by~(C3).

The sector $\Sigmalow_A$ is subdivided by the above 2-segments into $d+1$
polygons $D_1,\ldots,D_{d+1}$. Among these, $d$ polygons are lower semi-lenses,
each being associated with the segment $e_{p(i)}$ forming its upper boundary.
The remaining polygon contains the vertex $A$ and is viewed as a lower sector
$\Sigma'$ whose lower boundary $L(\Sigma')$ is formed by the segments
$e_{p(i)}$ with $P_{p(i)}$ maximal and the edges of $L_A$ between these
segments. We fill $\Sigma'$ with the corresponding $\Delta$-tiles (so each edge
$(B,C)$ of $L(\Sigma')$ generates one $\Delta$-tile, namely $(A|BC)$). We call
them \emph{secondary} $\Delta$-tiles.

In view of (C4), each 1-segment $e_i$ cutting $\Sigmalow_A$ connects some
vertex of $L(\Sigma')$ and the top $A$, and therefore $e_i$ coincides with the
common V-edge of some two neighboring $\Delta$-tiles in the filling of
$\Sigma'$.

A possible splitting of a sector by three 2-segments and one 1-segment (drawn
by dotted lines) is illustrated in Fig.~\ref{fig:split_sec}.

\begin{figure}[htb]
\vspace{0cm}
\begin{center}
\includegraphics{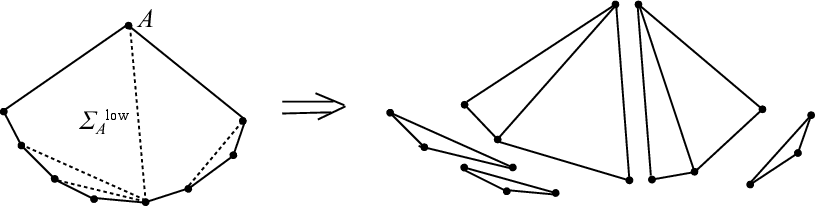}
\end{center}
\vspace{-0.5cm}
 \caption{Splitting a sector}
 \label{fig:split_sec}
  \end{figure}

When $\zeta_\Sscr$ cuts some $\nabla$-tile $\nabla(A|BC)$, we consider the
upper sector $\Sigmaup_A$ and make a splitting in a similar way (which
subdivides $\Sigmaup_A$ into corresponding upper semi-lenses and secondary
$\nabla$-tiles).

Let $\hat K$ be the resulting set of tiles upon termination of the splitting
process for $K$ w.r.t. $\zeta_\Sscr$ (it consists of the tiles of $K$ not cut
by $\zeta_\Sscr$ and the appeared semi-lenses and secondary tiles). Then the
desired $\Kin$ ($\Kout$) is defined to be the set of tiles of $\hat K$ lying in
$\Rin$ (resp. $\Rout$). We refer to $\hat K$, $\Kin$, $\Kout$ as
\emph{quasi-combies} on $Z$, $\Rin$, $\Rout$, respectively, \emph{agreeable
with} $\Sscr$. \smallskip

Now we are ready to generalize Theorem~\ref{tm:cyc_pattern}. As before, the
domain $\Din$ ($\Dout$) consists of the sets (points) $X\subseteq[n]$ such that
$X\weak \Sscr$ and $X$ lies in $\Rin$ (resp. $\Rout$).
    \begin{theorem} \label{tm:gen_cyc_pattern}
Let $\Sscr$ be a generalized cyclic pattern satisfying (C1)--(C4). Then the
domains $\Din$ and $\Dout$ form a complementary pair. As a consequence, both
$\Din$ and $\Dout$ are w-pure.
 \end{theorem}
  \begin{proof}
Given arbitrary $X\in\Din$ and $Y\in\Dout$, take a combi $K$ with $V_K$
including $\Sscr\cup\{X\}$ and a combi $K'$ with $V_{K'}$ including
$\Sscr\cup\{Y\}$. Split $K$ into the corresponding quasi-combies $\Kin$ and
$\Kout$ on $\Rin$ and $\Rout$, respectively, and similarly, split $K'$ into
quasi-combies $\Kpin$ and $\Kpout$. Then $X$ is a vertex in $\Kin$, and $Y$ is
a vertex in $\Kpout$. Their union $\tilde K:=\Kin\cup \Kpout$ is a quasi-combi
on $Z$. We transform $\tilde K$, step by step, in order to obtain a correct
combi with the same vertex set.

More precisely, in a current $\tilde K$ choose a semi-lens $\lambda$. If
$\lambda$ is a lower semi-lens with $\ell_\lambda=X$ and $r_\lambda=Y$, then
its upper boundary is formed by the unique edge $e=(X,Y)$, and this edge
belongs to another tile $\tau$ in $\tilde K$. Three cases are possible: (a)
$\tau$ is a $\Delta$-tile (and $e$ is its base); (b) $\tau$ is a lens or a
lower semi-lens (and $e$ belongs to its lower boundary $L_\tau$); and (c)
$\tau$ is an upper semi-lens (and $e$ forms its lower boundary). We remove the
edge $e$, combining $\lambda$ and $\tau$ into one polygon $\rho$.

In case (a), $\rho$ looks like an upper sector, and we fill it with the
corresponding $\Delta$-tiles. In case~(b), $\rho$ is again a lens or a lower
semi-lens (like $\tau$). In case~(c), $\rho$ is a lens. The new $\tilde K$ is a
quasi-combi on $Z$, and the number of semi-lenses becomes smaller.

If $\lambda$ is an upper semi-lens, it is treated symmetrically.

We repeat the procedure for the current $\tilde K$, and so on, until we get rid
of all semi-lenses. Then the eventual $\tilde K$ is a correct combi containing
both $X,Y$, and the result follows.
  \end{proof}


\subsection{Planar graph patterns} \label{ssec:graph_pat}

We can further extend the w-purity result by considering an arbitrary graph
$\Hscr=(\Sscr,\Escr)$ with the following properties: \smallskip

\noindent (H1) the vertex set $\Sscr$ is a w-collection in $2^{[n]}$;
\smallskip

\noindent (H2) each edge $e\in \Escr$ is formed by a 1- or 2-distance pair in
$\Sscr$; \smallskip

\noindent (H3) the edges of $\Hscr$ obey (C3) and (C4), in the sense that there
are no quadruple of vertices of $\Hscr$ that can be labeled as
$S_{p-1},S_{p},S_{q-1},S_q$ so that both $\{S_{p-1},S_p\}$ and
$\{S_{q-1},S_q\}$ are edges of $\Hscr$ and they behave as indicated in~(C3)
or~(C4).
\smallskip

Representing the vertices of $\Hscr$ as corresponding points in $Z=Z_n$, and
the edges as line segments, we observe from~(H3) and the proof of
Proposition~\ref{pr:non-self-int} that the graph $\Hscr$ is planar (has a
planar layout in $Z$). Let $\Fscr$ be the set of its (closed 2-dimensional)
faces. (W.l.o.g., we may assume that $\Hscr$ includes the entire boundary of
$Z$, since $bd(Z)$ is weakly separated from any subset of $[n]$.) For a face
$F\in \Fscr$, the set of elements of $\Dscr_\Sscr$ contained in $F$ is denoted
by $\Dscr_\Sscr(F)$.

   \begin{theorem} \label{tm:gen-gen}
Let $\Hscr$ be a graph satisfying (H1)--(H3). Then for any two different faces
$F,F'$ of $\Hscr$, the domains $\Dscr_\Sscr(F)$ and $\Dscr_\Sscr(F')$ form a
complementary pair. As a consequence, $\Dscr_\Sscr(F)$ is w-pure for each face
$F$, and similarly for any union of faces in $\Hscr$.
  \end{theorem}

(When $\Hscr$ is a simple cycle, this turns into
Theorem~\ref{tm:gen_cyc_pattern}. Also this generalizes the result on
semi-simple cyclic patterns in Sect.~\SSEC{semi-simple}.)
\smallskip

  \begin{proof}
If sets (points) $X,Y$ lie in $F,F'$, respectively, then they are separated by
the curve corresponding to some simple cycle $\Cscr$ in $\Hscr$. This $\Cscr$
is, in fact, a generalized cyclic pattern obeying conditions~(C1)--(C4). Also
one of $X,Y$ lies in the region $R_\Cscr^{\rm in}$, and the other in
$R_\Cscr^{\rm out}$. Now the result follows from
Theorem~\ref{tm:gen_cyc_pattern}.
  \end{proof}

One can reformulate this theorem as follows: For each face $F$ of $\Hscr$, take
an arbitrary maximal w-collection $\Xscr_F$ in $\Dscr_\Sscr(F)$; then for any
set $\Fscr'$ of faces of $\Hscr$, ~$\cup(\Xscr_F\colon F\in \Fscr)$ is a
maximal w-collection in $\cup(\Dscr_\Sscr(F) \colon F\in\Fscr')$.

Theorem~\ref{tm:gen-gen} was originally stated and proved in~\cite{DKK4}. An
alternative proof can be given based on nice properties of mutations of
w-collections in a discrete Grassmannian, established in a subsequent work of
Oh and Speyer~\cite{OS}.

We conclude this paper with a special case of generalized cyclic patterns,
namely a \emph{Grassmann necklace} of~\cite{OPS}. This is a sequence
$\Nscr=(S_1,S_2,\ldots, S_n=S_0)$ of sets in $\Delta_n^m$ such that
$S_{i+1}-S_i=\{i\}$ for each $i$. One can check that $\Nscr$ is a w-collection
satisfying~(C3). As is shown in~\cite{OPS}, the domain $\Dscr^{\rm in}_\Nscr$
is w-pure. A sharper result in~\cite{DKK3} says that the domains $\Dscr^{\rm
in}_\Nscr$ and $\Dscr^{\rm out}_\Nscr\cap \Delta_n^m$ form a complementary pair
within the hyper-simplex $\Delta_n^m$. This is a special case of
Theorem~\ref{tm:gen-gen}. Indeed, take as $\Hscr$ the union of the (natural)
cycle $\Cscr$ on $\Nscr$ and the cycle $\Cscr_0$ on the ``maximal'' necklace
(formed by the intervals and co-intervals if size $m$ in $[n]$). Then $\Hscr$
has the face surrounded by $\Cscr$ (giving the domain $\Dscr^{\rm in}_\Nscr$)
and the face (or the union of several faces) ``lying between'' $\Cscr$ and
$\Cscr_0$ (giving the domain $\Dscr^{\rm out}_\Nscr\cap \Delta_n^m$). \medskip

\noindent\textbf{Acknowledgements} ~We thank the anonymous referees for useful
remarks and suggestions. Supported in part by grant RSF 16-11-10075.



\begin{thebibliography}{99}

 %
\bibitem{DKK1} V.I.~Danilov, A.V.~Karzanov and G.A.~Koshevoy, Pl\"ucker environments,
wiring and tiling diagrams, and weakly separated set-systems,
\textsl{Adv.~Math.} \textbf{224} (2010) 1--44.
 %
\bibitem{DKK2} V.I.~Danilov, A.V.~Karzanov and G.A.~Koshevoy, On maximal
weakly separated set-systems, \textsl{J. Algebr. Comb.} \textbf{32} (2010)
497--531. See also \textsl{arXiv}:0909.1423v1[math.CO] (2009).
 %
\bibitem{DKK3} V.I.~Danilov, A.V.~Karzanov and G.A.~Koshevoy, The purity of
separated set-systems related to Grassmann necklaces,
\textsl{arXiv}:1312.3121[math.CO] (2013).
 %
\bibitem{DKK4} V.I.~Danilov, A.V.~Karzanov and G.A.~Koshevoy,
Combined tilings and the purity phenomenon on separated set-systems,
\textsl{arXiv}:1401.6418 [math.CO] (2014).
 %
\bibitem{LZ} B.~Leclerc and A.~Zelevinsky: Quasicommuting families of
quantum Pl\"ucker coordinates, \textsl{Am. Math. Soc. Trans. Ser.~2} ~{\bf 181}
(1998) 85--108.
 %
\bibitem{OPS} S.~Oh, A.~Postnikov, and D.E.~Speyer, Weak separation and plabic
graphs, {\textsl arXiv}:1109.4434[math.CO] (2011).
 %
\bibitem{OS} S.~Oh and D.E.~Speyer, Links in the complex of weakly separated collections,
{\textsl arXiv}:1405.5191[math.CO] (2014).
 %
\bibitem{Post} A.~Postnikov, Total positivity, Grassmannians, and networks,
\textsl{arXiv}:math.CO/0609764 (2006).

\end{thebibliography}
 \end{document}